\documentclass[12pt]{article}
\usepackage{amssymb}
\usepackage{makecell}
\usepackage{graphicx}
\usepackage{subcaption}
\usepackage{caption} 
\usepackage{array}
\usepackage{multirow}
\usepackage{amsmath}
\usepackage{multirow}
\usepackage{mathrsfs}
\usepackage{amsfonts}
\usepackage{indentfirst}
\usepackage{colortbl,dcolumn}
\usepackage{color}

\usepackage{amsmath}
\allowdisplaybreaks[4]
\usepackage{psfrag}

\usepackage{array}
\usepackage{booktabs}
\usepackage{cite}
\usepackage[hidelinks]{hyperref}
\usepackage{amsthm}
\allowdisplaybreaks
\numberwithin{equation}{section}
\usepackage[top=1in, bottom=1in, left=0.8in, right=0.8in]{geometry}

\usepackage{caption}
\usepackage{appendix} 
\usepackage{graphicx}
\usepackage{subcaption}
\usepackage{float}
\usepackage{makecell}
\newcommand{\R}{\mathbb{R}}
\newcommand{\N}{\mathbb{N}}
\newcommand{\E}{\mathbb{E}}

\newcommand{\dd}{\mathrm{d}}
\newtheorem{thm}{Theorem}[section]

\newtheorem{lemma}[thm]{Lemma}
\newtheorem{proposition}[thm]{Proposition}

\newtheorem{assumption}[thm]{Assumption}

\begin{document}
\bibliographystyle{plain}
\title{Accelerated Schr\"odinger-F\"ollmer samplers
\footnotemark[2] {}}
	\author{
		Haotian Lin, 
        Xiaojie Wang, 
        Xiaoyan Zhang 
		\\
		\footnotesize School of Mathematics and Statistics,   HNP-LAMA, Central South University, Changsha 410083, Hunan, China
		%
		%
		%
	}

\maketitle
\footnotetext[2]{
This work was supported by Natural Science Foundation of China (12471394, 12371417) and Hunan Basic Science Research Center for Mathematical Analysis (2024JC2002). 
We also gratefully acknowledge the partial computational support provided by the High Performance Computing Center of Central South University.
E-mail addresses: 
h.t.lin@csu.edu.cn,
x.j.wang7@csu.edu.cn,
x.y.zhang@csu.edu.cn.
 }

\begin{abstract}
Sampling is a fundamental algorithmic task in wide-ranging applications across multiple disciplines such as scientific computing, statistics and machine learning.
In this paper, an efficient stochastic Runge-Kutta scheme is proposed to accelerate the Schr\"odinger-F\"ollmer sampler, 
designed for sampling from complex and high-dimensional multimodal distributions.
The resulting stochastic Runge-Kutta Schr\"odinger-F\"ollmer sampler (SRKSFS) is proved to achieve a convergence rate of order $\mathcal{O} ( h^{3/2} |\ln h|)$ in the $L^2$-Wasserstein distance,
considerably improving the order $\mathcal{O}(h)$ of the existing Euler type sampler.
Obtaining the enhanced convergence rate is, however, not trivial, by noting that the drift of the diffusion process is not differentiable but only $\frac{1}{2}$-H\"older continuity with respect to the time variable. To address the difficulty, we rely on delicate error estimates to overcome the singularity due to time derivatives of the drift, at the expense of the logarithmic factor.
Furthermore, the framework is extended to data-driven Schr\"odinger-F\"ollmer generation with empirical measures, enabling data-driven sampling without known density.
A variety of numerical experiments are reported to validate the effectiveness of the proposed sampling algorithms. 

\end{abstract}

\section{Introduction}
\noindent
Sampling from complex and high-dimensional probability distributions is a fundamental task in computational statistics and machine learning,
with critical applications in
Bayesian inference \cite{gelman1995bayesian}, 
generative modeling \cite{yang2023diffusion,bond2021deep}, 
uncertainty quantification \cite{sullivan2015introduction}, 
computer vision \cite{huang2022multimodal}
and drug discovery\cite{zeng2022deep}.
Broadly, the sampling problem can be categorized by the information available about the target distribution $\mu$: (i) the setting where its density $\mu(\text{d}x) \propto e^{-V(x)}\text{d}x$ is explicitly known, albeit potentially complex, and (ii) the data-driven setting where the density is unknown but independent samples from $\mu$ are accessible.

In the case when the density is known, a widely used sampling algorithm is the (overdamped) Langevin Monte Carlo (LMC)  (also called unadjusted Langevin algorithm, ULA)
\cite{altschuler2024shifted,dalalyan2017theoretical, durmus2017nonasymptotic, li2025sharp, li2021sqrt, mou2022improved, yang2025non, yang2026accelerating,li2019stochastic}, 
based on Euler or Runge-Kutta discretizations of Langevin stochastic differential equations (SDEs):
\begin{equation}
\mathrm{d} X_{t} = -\nabla V (X_{t}) \mathrm{d} t + \sqrt{2} \mathrm{d} W_{t},
\quad X_{0} = x_{0}, \quad t > 0,
\end{equation}
where the drift term $-\nabla V (\cdot)$ is a log-gradient of the target density $\mu(\text{d}x) \propto e^{-V(x)}\text{d}x$,  
and $\{W_t\}_{t\geq 0}$ is a standard $d$-dimensional Brownian motion.
Under suitable conditions on
$-\nabla V$, the overdamped Langevin SDE is ergodic 
and admits the target distribution $\mu$
as its unique invariant distribution. 
Consequently, sampling from $\mu$ can be done by simulating this SDE over a long time horizon. 

In the context of strong convexity, non-asymptotic error bounds of LMC have been well-established in various distances
\cite{dalalyan2017theoretical, durmus2017nonasymptotic, durmus2019high}.
Since the strong convexity condition is often too restrictive in practice, the non-asymptotic error analysis of LMC has been recently investigated in non-convex settings such as contractivity at infinity
\cite{cheng2018sharp,durmus2017nonasymptotic,eberle2016reflection, pang2025projected, schuh2024convergence}
and the log-Sobolev inequality
\cite{li2025sharp, mou2022improved,yang2025non}.
Nevertheless, 
LMC relies on ergodicity and often has slow mixing in practice, 
especially for high-dimensional or multimodal distributions, where samples may become trapped in local modes, 
failing to capture the full probabilistic mass of the target 
\cite{qiu2024efficient, hale2010asymptotic}.

An alternative paradigm, which circumvents the need for long-time integration and aims to address these limitations, seeks to construct a diffusion process that transports a simple initial distribution (e.g.,  a point mass or Gaussian) to the target $\mu$ within a finite time interval $[0,1]$. This framework is central to diffusion-based generative models and Schr{\"o}dinger bridge problems.
Significant efforts have been made along this direction, see, e.g.,
\cite{albergo2023building,albergo2025stochastic,dai2023global,huang2025schrodinger,ruzayqat2023unbiased,wang2025multimodal,huang2024one}.
Given the target distribution $\mu \in \mathcal{P} (\mathbb{R} ^d)$ 
and a simple probability distribution $\nu \in \mathcal{P} (\mathbb{R} ^d)$, 
the general problem can be formulated as constructing and solving an SDE of the form
\begin{equation}
    \mathrm{d} X_{t}=b\left(t, X_{t}\right) \mathrm{d} t+\sigma\left(t, X_{t}\right) \mathrm{d} W_{t}, \quad X_{0} \sim \nu, \quad X_{1} \sim \mu, \quad t \in(0,1].
\end{equation}
%
However, finding the closed-form of the drift and diffusion coefficients is usually not an easy task
(cf.\cite{albergo2023building,albergo2025stochastic}).
In \cite{follmer2005entropy,follmer2006random}, 
F{\"o}llmer proposed a Schr{\"o}dinger-F{\"o}llmer diffusion process in the context of the Schr{\"o}dinger bridge problem \cite{schrodinger1932theorie}, 
connecting two distributions. 
More precisely,
for a target distribution $\mu (\mathrm{d} x) \propto  e^{-V(x)}\mathrm{d} x$, 
the Schr{\"o}dinger-F{\"o}llmer diffusion evolving on the unit
interval $[0, 1]$ is given by:
\begin{equation}
\label{eq:sfp}
    \dd X_t = \nabla \log 
    \big( Q_{1-t} g(X_t) \big)
    \dd t 
    + 
    \dd W_t, 
\quad X_0 = 0 ,  
\quad t \in [0,1], 
\end{equation}
where $g$ denotes the Radon-Nikodym derivative of $\mu$ 
with respect to the Gaussian measure $\mathcal{N} (0, \mathbf{I}_d)$
\begin{equation}
g(x) 
:= 
\tfrac{\mathrm{d}\mu}{\mathrm{d}\mathcal{N}(0, 
 \mathbf{I}_{d})}(x) 
 =
 \tfrac{(2\pi)^{d/2}}{C} \exp
 \big( -V(x)
 + \tfrac{\|x\|^{2}}{2} \big), 
\quad x \in \mathbb{R}^{d}, 
\end{equation}
and $\{ Q_{t} \}_{t \in [0,1]}$ denotes the heat semigroup defined by
\begin{equation}
Q_{t} g(x)
:= 
\mathbb{E}_{\xi} 
\big[ g(x + \sqrt{t}  \xi)
\big],
\quad t \in [0,1], 
\quad \xi \sim \gamma^{d}. 
\end{equation}
The diffusion process \eqref{eq:sfp} transports the degenerate distribution $\delta_0$ at $t=0$ to the target distribution 
$
\mu(\mathrm{\,d} x) 
    \propto 
     e^
     {-V(x)} 
    \mathrm{\,d} x
$
at $t=1$. 
Recently, the authors of \cite{huang2025schrodinger} proposed a Schr\"odinger-F\"ollmer sampler (SFS) based on the Euler discretization of SDE \eqref{eq:sfp} with a uniform time step size $h>0$. Moreover, they proved a convergence rate of order $\mathcal{O}(\sqrt{dh})$ in the $L^2$-Wasserstein distance. 
%
More recently, a variant of Schr\"odinger-F\"ollmer diffusion with temperatures is introduced in \cite{wang2025multimodal}:
\begin{equation}
\label{eq:sfp-temp}
\dd 
X_{t} 
=
\nabla \log
\big(
Q^\beta_{1-t} g_\beta(X_t) 
\big)
\dd t
+ 
\sqrt{\beta}  
\,
\dd W_{t}, 
\quad X_{0} = 0, 
\quad t \in (0,1],
\end{equation}
where $g_\beta$ and $Q^\beta_t, \, \beta \in (0, \infty)$ are determined by \eqref{eq-sfs-g-beta} and \eqref{eq-sfs-Qt-beta}, respectively.
Building on the Euler discretization of \eqref{eq:sfp-temp}, new sampling algorithms are constructed, achieving an enhanced convergence rate of order $\mathcal{O}(dh)$. 
%
An interesting question thus arises:
\\

{\it Can an efficient sampling algorithm with higher-order convergence be constructed, based on a higher-order time discretization scheme for the Schr\"odinger-F\"ollmer diffusion?}
\\

This is, however, not trivial, by noting that the drift of the diffusion process \eqref{eq:sfp-temp} is not differentiable but only $\frac{1}{2}$-H\"older continuity with respect to the time variable (see \eqref{eq:b-lipschitz-in-x-houder-in-t}). 
In this work, we attempt to answer the question in the affirmative. 
Inspired by the idea of stochastic Runge-Kutta methods for high-order strong approximation of SDEs \cite{milstein2004stochastic,kloeden1992numerical}, we design a stochastic Runge-Kutta Schr\"odinger-F\"ollmer sampler (SRKSFS) as follows:
\begin{align}
\begin{split}
        Y_{{n+1}} 
        & = Y_{n} 
        + \tfrac{1}{3}
        f_\beta(t_n,Y_{n})h 
        + \tfrac{2}{3}
        f_\beta(t_n+\tfrac{3}{4}h,H_{{n}})h 
        + \sqrt{\beta}
        \Delta W_n,
\\
H_{{n}} 
        & = Y_{n} 
        + \tfrac{3}{4}f_\beta(t_n,Y_{n})h 
        + \tfrac{3\sqrt{\beta}
        \Delta Z_n}{2h}, 
\end{split}
\end{align}
where $f_{\beta}(t, x) := \beta  
\,
\nabla \log \big( Q_{1-t}^{\beta} g_{\beta}(x) \big)$ is the drift of \eqref{eq:sfp-temp}. Here $\Delta W_n, \,\Delta Z_n $ defined by \eqref{eq-srk-Wn-Zn} are both Gaussian
and can be generated by two sequences of independent standard Gaussian random variables (cf. \eqref{eq:deltaW-deltaZ-implement}).
For SRKSFS, we establish a non-asymptotic convergence rate of order $\mathcal{O}((dh)^{3/2}|\ln h|)$ in the $L^2$-Wasserstein distance.
In the classical convergence analysis, higher convergence rates are obtained at the price of sufficient smoothness of coefficients of SDE \cite{kloeden1992numerical,milstein2004stochastic}.
However, due to the limited temporal H\"older  regularity of the drift of \eqref{eq:sfp-temp}, obtaining the high convergence rate turns out to be highly non-trivial.
This difficulty is overcome by a more refined analysis in handling the singularity in time derivatives of the drift (see the proof of Theorem \ref{thm-exact-drift}).

The above mentioned convergence rates were obtained when the exact drift coefficient of \eqref{eq:sfp-temp} was used. 
However, in most cases when the target distribution is complex, the exact computation of the drift in \eqref{eq:sfp-temp} involving an expectation is intractable. Then we have to approximate the exact drift $f_{\beta}$ in the above sampler with a Monte Carlo approximation $\widetilde{f}_{\beta}^{M}$ given by \eqref{eq-sfs-mc-exact}, resulting in the other sampler \eqref{eq-srk-mc-appr}. 
In this context, the sampling error arises from both the time discretization and the Monte Carlo approximation of the drift.
The latter approximation error is characterized by $\mathcal{O}\Big(\sqrt{\tfrac{d}{M}}\Big)$, where $M$ is the number of samples used in the Monte Carlo approximation (cf. Theorem \ref{thm-mc-error-bound}). 
%

%
%
In addition, we also extend the accelerated Schr\"odinger-F\"ollmer sampler to handle the task of sampling in the case when the target distribution $\mu$ is unknown. 
Based on available independent samples  $\{\eta^{(i)}\}^{n}_{i=1} \sim \mu$, we propose a data-driven sampler \eqref{eq:sfs-data} to generate new samples from the unknown distribution $\mu$. Image generation tasks using MNIST and CIFAR-10 datasets demonstrate its superior performance in data-driven generation (see Section \ref{sec:Data-driven-generation} for more details).

In summary, 
our contributions are as follows:
\begin{itemize}
\item 
An accelerated Schr\"odinger-F\"ollmer sampler \eqref{eq-srk-appr} 
is proposed, with the following error bound in $L^2$-Wasserstein distance established:
\[
\mathcal{W}_2 
        \big(
        \text{Law}\,(Y_{1}), \mu
        \big) 
        \leq
        C
        (dh)^{3/2}
        |\ln h|.
\]

\item 
When the exact computation of the drift is intractable, we introduce the sampler \eqref{eq-srk-mc-appr} with an inexact drift due to Monte Carlo approximation and obtain the following
error bound:
\[
\mathcal{W}_2
\big(
\text{Law}(\widetilde{Y}_1), \mu
\big) 
\leq 
C (d h)^{3/2} |\ln h| 
+
C \sqrt{\tfrac{d}{M}},
\]
where $M$ is the number of samples used in the Monte Carlo estimator of the drift term.

\item 
When the target distribution $\mu$ is unknown, but only accessible through empirical samples  $\{\eta^{(i)}\}^{n}_{i=1} \sim \mu$, 
we introduce a data-driven sampler \eqref{eq:sfs-data} to generate new samples from the unknown distribution $\mu$.

\item 
A variety of numerical experiments are reported to validate the effectiveness of SRKSFS for sampling from 
known densities and demonstrate its superior performance in data-driven generation from empirical measures when the densities are unknown.
\end{itemize}

The structure of this paper is as follows. 
Section \ref{sec:sfp} presents the Schr\"odinger-F\"ollmer diffusion with temperatures.  
Accelerated Schr\"odinger-F\"ollmer samplers are introduced in Section \ref{sec:sfs} with error bounds obtained. 
In Section \ref{sec:experiments}, numerical experiments are reported on sampling from Gaussian mixtures, 
copula-generated distributions, and deep generative models. 
Section \ref{sec:Data-driven-generation} extends the framework to data-driven Schr\"odinger-F\"ollmer generation using empirical measures and reports some numerical experiments. 
Finally, Section \ref{sec:Conclusion} concludes this work and discusses future directions.

\section{Schr\"odinger-F\"ollmer diffusion with temperatures}
\label{sec:sfp}
\subsection{Notation}
Throughout this paper, 
we use $\mathbb{N}$ for the set of all positive integers 
and let $d \in \mathbb{N}$. 
For short, 
we denote $[d] := \{1,\cdots,d\}$ and   
$[d]_0:= \{0, 1,\cdots,d\}$.
Let $\|\cdot\|$  and $\langle\cdot,\cdot\rangle$ denote the Euclidean norm 
and the inner product of vectors in $\mathbb{R}^{d}$, respectively. 
We use $\mathbf{1}_d \in \mathbb{R}^{d}$ 
and $\mathbf{I}_d \in \mathbb{R}^{d \times d}$ 
to denote the all-ones vector (where all entries are $1$) and the identity matrix, respectively.
For any matrix $A =(a_{i,j})\in \mathbb{R}^{d\times d}$, the $\mathrm{F}$-norm is defined as
$\|A\|_{\mathrm{F}}
:=
\sqrt{\sum_{i,j=1}^d |a_{i,j}|^2}$, 
and the operator norm as $\|A\| =\|A\|_{\mathrm{op}} :=\sup_{\|v\|=1}\|Av\|$. 
It is not difficult to see
$
\|A\|
\leq
\|A\|_{\mathrm{F}}
\leq 
\sqrt{d}
\,
\|A\|.
$

Let $\left\{W_t\right\}_{t \in [0, 1] }$ be a standard $d$-dimensional Brownian motion process, defined on a filtered probability space $
\big(
\Omega_W, 
\mathcal{F}^W,
\mathbb{P}_W, 
\{\mathcal{F}^W_t\}_{t \in[0,1]}
\big)
$ 
satisfying the usual conditions.
Let $\{\xi_{i}\}_{ i \in \mathbb{N}}$ be an i.i.d. family of standard Gaussian distributed random variables, independent of $\left\{W_t\right\}_{t \in [0, 1]}$,  defined on an additional filtered probability space 
$
\left(
\Omega_\xi, 
\mathcal{F}^\xi,
\mathbb{P}_\xi\right)$.
%
By $\mathbb{E}_W$ and $\mathbb{E}_\xi$, we denote expectations in these two probability spaces 
$
\big(
\Omega_W, 
\mathcal{F}^W,
\mathbb{P}_W
\big)
$
and
$
\left(
\Omega_\xi, 
\mathcal{F}^\xi,
\mathbb{P}_\xi\right),
$
respectively.
By introducing the following product probability space
\begin{align*}
    (\Omega,\mathcal{F},\mathbb{P},\mathcal{F}_t)
    :=
    (\Omega_W\otimes\Omega_{\xi}, \mathcal{F}^W
    \otimes
    \mathcal{F}^{\xi},
    \mathbb{P}_W
    \otimes\mathbb{P}_{\xi},
    \mathcal{F}^W_t 
    \otimes
    \mathcal{F}^{\xi}),
\end{align*}
%
%
we use $\mathbb{E} $ for the expectation in the product probability space and $L^r\left(\Omega, \mathbb{R}^d\right), r \geq 1$, to denote the family of $\mathbb{R}^d$-valued random variables $\eta$ satisfying $\mathbb{E} \left[\|\eta\|^r\right] := \mathbb{E}_W \big( \mathbb{E}_\xi \left[\|\eta\|^r\right] \big) <\infty$. 
Let $\mathcal{B}\left(\mathbb{R}^d\right)$ be the Borel $\sigma$-field of $\mathbb{R}^d$ and $\mathcal{P}\left(\mathbb{R}^d\right)$ be the space of all probability distributions on $\left(\mathbb{R}^d, \mathcal{B}(\mathbb{R}^d)\right)$.
By $\mathcal{L}aw(X)$ we denote the probability distribution of the $\mathbb{R}^d$-valued random variable $X$.
Let $\nu_1$ and $\nu_2$ be two probability measures defined on $\left(\mathbb{R}^d, \mathcal{B}(\mathbb{R}^d)\right)$, 
and let $\mathcal{D}\left(\nu_1, \nu_2\right)$ represent the collection of couplings $\nu$ on $\left(\mathbb{R}^{2 d}, \mathcal{B}(\mathbb{R}^{2 d})\right)$ whose first and second marginal distributions are $\nu_1$ and $\nu_2$, respectively.
The $L^2$-Wasserstein distance between $\nu_1$ and $\nu_2$ is defined by
\[
\mathcal{W}_2
(\nu_1, \nu_2)
:=
\inf_{\nu \in \mathcal{D}
(\nu_1, \nu_2)}
\left(
\int_{\mathbb{R}^d} \int_{\mathbb{R}^d}
\left\|
\theta_1-\theta_2
\right\|^2 
\mathrm{\, d} 
\,
\nu
\left(
\theta_1, \theta_2
\right)
\right)^{1/2}.
\]
Let $\mathcal{C}\left(\mathbb{R}^d, \mathbb{R}\right)$ denote the set of all continuous functions from $\mathbb{R}^d$ to $\mathbb{R}$ and let $\mathcal{C}_b\left(\mathbb{R}^d, \mathbb{R}\right)$ denote the set of all bounded continuous functions from $\mathbb{R}^d$ to $\mathbb{R}$. 
For $k \geq 0$, denote by $\mathcal{C}^k\left(\mathbb{R}^d, \mathbb{R}\right)$ the set of functions from $\mathbb{R}^d$ to $\mathbb{R}$ which have continuous $0$-th, $\ldots, k$-th order derivatives, and denote by $\mathcal{C}_b^k\left(\mathbb{R}^d, \mathbb{R}\right)$ the set of functions from $\mathbb{R}^d$ to $\mathbb{R}$ which have bounded continuous $0$-th, $\ldots, k$-th order derivatives.
For $y \in \mathcal{C}^3(\mathbb{R}^d, \mathbb{R})$ and $v_1, v_2, v_3, x \in \mathbb{R}^d$, we denote
\begin{equation}
\begin{aligned}
\label{SFS-W2-eq:def-dir-der}
        \nabla_{v_1} y(x) 
        & =
        \lim_{\varepsilon \rightarrow 0} \frac{y(x+\varepsilon v_1)
        -y(x)}
        {\varepsilon}, 
        \\
        \nabla_{v_2} 
        \nabla_{v_1} y(x) 
        & =
        \lim_{\varepsilon \rightarrow 0} \frac{\nabla_{v_1} y\left(x+\varepsilon v_2\right)
        -
        \nabla_{v_1} y(x)}{\varepsilon},\\
        \nabla_{v_3} 
        \nabla_{v_2} 
        \nabla_{v_1} y(x)
        & =
        \lim_{\varepsilon \rightarrow 0} \frac{\nabla_{v_2} \nabla_{v_1} y\left(x+\varepsilon v_3\right)
        -\nabla_{v_2} 
        \nabla_{v_1} y(x)}{\varepsilon},
\end{aligned}
\end{equation}
as the directional derivatives of $y$. 
If the function $y$ is differentiable at the point 
$x \in \mathbb{R}^d$, then the directional derivative exists along any nonzero vector $v\in \mathbb{R}^d$. 
In this case, we have
\begin{align*}
    \nabla_{v}y(x)=
    \left \langle
    \nabla y(x), v
    \right\rangle.
\end{align*}
One knows 
$\nabla y(x) \in \mathbb{R}^d, \nabla^2 y(x) \in \mathbb{R}^{d \times d}, 
\nabla^3 y(x) \in$ $\mathbb{R}^{d \times d \times d}$. 
Moreover, we define the operator norm of $\nabla^k y(x), k=1,2,3$ by
\begin{align}
\label{SFS-W2-eq:operator-norm-def}
    \big\|
    \nabla^k y(x)
    \big\|
    =
    \big\|
    \nabla^k y(x)
    \big\|_{\mathrm{op}}
    :=
    \sup_{\|v_i\|=1, i=1, \ldots, k}
    \left\|
    \nabla_{v_k} 
    \ldots 
    \nabla_{v_1} y(x)
    \right\|.
\end{align} 
For the vector-valued function $\mathbf{u}: \mathbb{R}^d \rightarrow \mathbb{R}^{\ell}, \mathbf{u}=\left(u_{(1)}, \ldots, u_{(\ell)}\right)^{\mathrm{T}}$, we regard its first order partial derivative as the Jacobian matrix:
\begin{align*}
    D \mathbf{u}=\left(\begin{array}{ccc}
\frac{\partial u_{(1)}}{\partial x_1} & \cdots & \frac{\partial u_{(1)}}{\partial x_d} \\
\vdots & \ddots & \vdots \\
\frac{\partial u_{(\ell)}}{\partial x_1} & \cdots & \frac{\partial u_{(\ell)}}{\partial x_d}
\end{array}\right)_{\ell \times d} .
\end{align*}
For any $v_1 \in \mathbb{R}^d$, one knows $D(\mathbf{u}) v_1 \in \mathbb{R}^{\ell}$ and one can define $D^2 \mathbf{u}\left(v_1, v_2\right)$ as
\begin{align*}
    D^2 \mathbf{u}\left(v_1, v_2\right):=D\left(D(\mathbf{u}) v_1\right) v_2, \quad \forall v_1, v_2 \in \mathbb{R}^d .
\end{align*}
In the same manner, one can define 
\[
 D^3 \mathbf{u}
 (v_1,v_2,v_3)
 :=
 D\big(
 D\left(D(\mathbf{u}) v_1\right) v_2\big)v_3, 
 \quad \forall v_1, v_2,v_3 \in \mathbb{R}^d .
\]

\subsection{The Schr{\"o}dinger-F{\"o}llmer process with temperatures}  
In this subsection,
we will recall the Schr{\"o}dinger-F{\"o}llmer process with temperatures. 
We begin with the following fundamental assumption.
\begin{assumption}
\label{ass-abso-cont}
Let the target distribution 
$\mu( \mathrm{d} x) \propto  \exp(-V(x))\mathrm{d} x $ over $x \in \mathbb{R}^d$, 
be absolutely continuous with respect to the $d$-dimensional 
Gaussian distribution denoted by $\mathcal{N} (0,\beta \,\mathbf{I}_d)$.    
\end{assumption}
Recall the Schr{\"o}dinger-F{\"o}llmer process with temperatures introduced in the latest  work \cite{wang2025multimodal}:
\begin{equation}
\label{eq-temp-SDE}
\dd 
X_{t} 
= f_{ \beta}(t,X_{t})  \dd t 
+ 
\sqrt{\beta}  
\,
\dd W_{t}, 
\quad X_{0} = 0, 
\quad t \in (0,1],
\end{equation}
where $\beta \in (0, \infty)$ serves as a temperature parameter
and the drift term $f_{\beta}$ is given by:
\begin{equation}
\label{eq-sfs-f-beta}
f_{\beta}(t, x) := \beta  
\,
\nabla \log \big( Q_{1-t}^{\beta} g_{\beta}(x) \big), 
\quad x \in \mathbb{R}^{d},  
\quad t \in [0,1].
\end{equation}
Here $g_\beta$ denotes the Radon-Nikodym derivative of $\mu$ 
with respect to the 
$\mathcal{N} (0,\beta \,\mathbf{I}_d)$:
\begin{equation}
\label{eq-sfs-g-beta}
g_{\beta}(x) := \tfrac{\mathrm{d}\mu}{\mathrm{d}\mathcal{N}(0, 
\beta \, \mathbf{I}_{d})}(x) = \tfrac{(2\pi\beta)^{d/2}}{C} \exp \left( -V(x) + \tfrac{\|x\|^{2}}{2\beta} \right), 
\quad x \in \mathbb{R}^{d}, 
\end{equation}
and $\{ Q_{t}^{\beta} \}_{t \in [0,1]}$ denotes the heat semigroup defined by
\begin{equation}
\label{eq-sfs-Qt-beta}
Q_{t}^{\beta} g_{\beta}(x) := \mathbb{E}_{\xi} \big[ g_{\beta}(x + \sqrt{t \beta} \, \xi) \big],
\quad t \in [0,1], 
\quad \xi \sim \gamma^{d},
\end{equation}
where $\gamma^{d}$ denotes the $d$-dimensional standard Gaussian distribution.
In light of Stein's lemma \cite[Lemma 3.6.5]{garthwaite2002statistical}, which enables us to avoid the calculation of $\nabla g_{\beta}$, one can obtain, for $t \in [0,1)$,
\begin{equation}
\mathbb{E}_{\xi} 
\big[ \nabla g_{\beta}(x + \sqrt{(1-t)\beta} \, \xi) \big] = \tfrac{1}{\sqrt{(1-t)\beta}} \mathbb{E}_{\xi} \big[ \xi \, g_{\beta}(x + \sqrt{(1-t)\beta}\,  \xi) \big].
\end{equation}
This identity allows us to recast the drift as a gradient-free form:
\begin{equation}
    \label{eq-sfs-exact-dens}
    \begin{aligned}
        f_{\beta}(t,x) &= \frac
            {\beta  \,\mathbb{E}_{\xi} 
            \big[ \nabla g_{\beta}
                (x + \sqrt{(1-t)\beta} \, \xi) 
            \big]}
            { \mathbb{E}_{\xi} 
            \big[ g_{\beta}(x + \sqrt{(1-t)\beta} \, \xi) 
            \big] } \\
            &= \frac
            {\beta\,  \mathbb{E}_{\xi} 
            \big[ \xi\,  g_{\beta}(x + \sqrt{(1-t)\beta} \, \xi) 
            \big]}
            { \mathbb{E}_{\xi} 
            \big[ g_{\beta}(x + \sqrt{(1-t)\beta} \, \xi) 
            \big] \cdot \sqrt{(1-t)\beta} }, 
            \quad \xi \sim \gamma^{d},
            \quad 
            t \in [0,1).
    \end{aligned}
\end{equation} 
In the special case $ \beta = 1 $, 
the temperature-dependent process reduces to the standard Schr\"odinger-F\"ollmer process, 
which has been previously studied in the literature \cite{huang2025schrodinger,dai2023global,ruzayqat2023unbiased}.
To ensure the well-posedness of the Schr{\"o}dinger-F{\"o}llmer diffusion process,
we make the following assumption.
\begin{assumption}
\label{ass-RN-lip}
Suppose that the Radon-Nikodym derivative $g_\beta$ and $\nabla g_\beta$ 
are $L_g$-Lipschitz continuous, 
and $g_\beta$ is uniformly bounded below by a positive constant $\rho$:
\begin{equation}
    \label{eq-RN-lip}
    g_\beta \geq \rho > 0.
\end{equation}
\end{assumption}

Under the above assumptions, 
we have the following well-posedness of the Schr{\"o}dinger-F{\"o}llmer diffusion process (\ref{eq-temp-SDE}) (see \cite{huang2025schrodinger,wang2025multimodal}).
\begin{proposition}
\label{prop-drift-lip}
    Let Assumptions \ref{ass-abso-cont}, \ref{ass-RN-lip} hold. 
    Then for any $t \in [0,1]$, 
    the drift coefficient $f_{\beta}(t, \cdot): \mathbb{R}^{d} \to \mathbb{R}^{d}$ 
    of SDE (\ref{eq-temp-SDE}) is Lipschitz continuous and of linear growth. 
    That is, there exist constants $L_{f}, \hat{L}_{f} > 0$, 
    independent of $d$, such that, for any $x, y \in \mathbb{R}^{d}$ and $t \in [0,1]$,
\begin{equation}
    \label{eq-lip-fx}
\|f_{\beta}(t,x) - f_{\beta}(t,y)\| \leq L_{f} \|x - y\|, 
\end{equation}
and
\begin{equation}
\begin{aligned}
    \label{eq-linear-growth}
\|f_{\beta}(t, x)\| 
\leq \|f_{\beta}(t,0)\|+  L_{f} \|x\| 
\leq \hat{L}_{f} + L_{f} \|x\|,
\end{aligned}
\end{equation}
where $L_f :=  \big(1+\tfrac{L_g}{\rho}
    \big)
    \tfrac{\beta L_g}{\rho},
    \,
    \hat{L}_f := \tfrac{\beta L_g}{\rho}$.
Then the Schr{\"o}dinger-F{\"o}llmer diffusion (\ref{eq-temp-SDE}) has a unique strong solution $\{X_{t}\}_{t \in [0,1]}$ satisfying $X_{1} \sim \mu$.
\end{proposition} 
As shown in \cite[Remark III.1]{huang2025schrodinger}, under Assumptions \ref{ass-abso-cont}, \ref{ass-RN-lip}, one can get $\tfrac{1}{2}$-H\"older continuity of the time-dependent drift with respect to the time variable:
\begin{align}
\label{eq:b-lipschitz-in-x-houder-in-t}
     \|f_{\beta}(t, x)-f_{\beta}(s, y)\|
     \leq 
     \bar{L}_f(\|x-y\|
     +
     d^\frac{1}{2}
     |t-s|^\frac{1}{2}).
\end{align}
Furthermore, we have the following moment estimates and $\tfrac{1}{2}$-H\"older continuity of the process $\{X_t\}_{t \in [0,1]}$, which can be found in Lemma \cite[Lemma 2.4]{wang2025multimodal}.

\begin{lemma}
\label{lem:moments-bounded}

Let Assumptions \ref{ass-abso-cont} and \ref{ass-RN-lip} hold. 
Then the solution of SDE (\ref{eq-temp-SDE}), 
denoted by $\{X_t\}_{t\in[0,1]}$ satisfies, 
for any $t\in[0,1]$ and $0\leq t_1\leq t_2\leq 1$,
\begin{equation}
\mathbb{E}_{W}
\left[
\|X_{t}\|^{2}\right] \leq M_1 d, 
\end{equation}
and
\begin{equation}
\mathbb{E}_{W}\left[\|X_{t_{2}}-X_{t_{1}}\|^{2}\right] \leq M_2 d | t_2 - t_1|,
\end{equation}
where $M_1 := 2\big(2\hat{L}_{f}^{2}+\beta\big) \exp\left(4L_{f}^{2}\right)$
and $M_2 :=  8L_{f}^{2} \exp(4L_{f}^{2}) (2\hat{L}_{f}^{2}+\beta) + 4\hat{L}_{f}^{2} + 2\beta$,
with constants $\hat{L}_{f}$ and $L_{f}$ coming from (\ref{eq-lip-fx}) and (\ref{eq-linear-growth}).
\end{lemma}

\section{Accelerated Schr\"odinger-F\"ollmer samplers}
\label{sec:sfs}
\noindent
One can sample from the target distribution 
$\mu$ by solving SDE \eqref{eq-temp-SDE}. 
In practice, this continuous-time process needs to be discretized to produce tractable numerical approximations.

\subsection{Accelerated Schr\"odinger-F\"ollmer sampler with exact drift}
The canonical approach to deriving high-order strong approximation schemes for SDEs is to employ It\^o Taylor expansions, as delineated in the classical works \cite{kloeden1992numerical,milstein2004stochastic}. 
For $N\in \mathbb{N} $, we define a uniform temporal partition of $[0,1]$ 
by $t_n = nh$ for $0 \leq n\leq N$, 
where $h=1/N$ stands for the uniform time step size. 
As established in \cite{milstein2004stochastic}, when the drift coefficient $f_{\beta}$ satisfy appropriate smoothness and the boundedness conditions, an order $1.5$ strong Taylor scheme for the SDE \eqref{eq-temp-SDE} can be constructed as follows:
\begin{equation}
\label{eq-sde-sfs-taylor}
\begin{aligned}
    Y_{n+1} 
    &= Y_n 
    + 
    f_\beta(t_n,Y_n) h 
    + \sqrt{\beta} \Delta W_n 
    +  \sqrt{\beta} Df_\beta(t_n,Y_n) \Delta Z_{n+1} \\
& \quad
    + 
    \tfrac{h^2}{2}
    \Big( \partial_t f_\beta(t_n,Y_n)
    + D f_\beta(t_n,Y_n) f_\beta(t_n,Y_n)
    +
    \tfrac{\beta}{2} \sum_{j=1}^d D^2 f_\beta(t_n,Y_n)[e_j,e_j]
    \Big),
\end{aligned}
\end{equation}
where both $\Delta W_n$ and $\Delta Z_n$  are Gaussian and given by
\begin{equation}
\label{eq-srk-Wn-Zn}
    \Delta W_n = \int_{t_n}^{t_{n+1}} \mathrm{d} W_s,
    \quad
    \Delta Z_n 
    = \int_{t_n}^{t_{n+1}}\int_{t_n}^{s} \mathrm{d} W_r
    \,
    \mathrm{d} s
    = \int_{t_n}^{t_{n+1}} 
    (t_{n+1}-r) 
    \,
    \mathrm{d} W_r.
\end{equation}
A primary drawback of the above scheme \eqref{eq-sde-sfs-taylor} lies in the need to compute higher-order derivatives of the drift coefficient at each step.
This requirement can incur significant computational costs in high-dimensional settings and may limit practical applicability.
Consequently, we seek a higher-order method that avoids complex derivative calculations.
Following the idea of Runge-Kutta methods, we propose a class of stochastic Runge-Kutta Schr\"odinger-F\"ollmer sampler (SRKSFS) algorithms that achieve the same order of accuracy while avoiding explicit derivative evaluations. 
More precisely, we introduce a specific SRK method of order $1.5$ as follows:
\begin{equation}
\label{eq-sfs-srk-general}
\begin{aligned}
    Y_{n+1} 
    &= 
    Y_n  
    + (1-\omega_1-\omega_2) f_\beta(t_n,Y_n) h
    + \omega_1 f_\beta(t_n + c_1 h,\Phi_n^1) h\\
    & \quad +
    \omega_2 f_\beta(t_n + c_2 h,\Phi_n^2) h
    +\sqrt{\beta} \Delta W_n,
\end{aligned}
\end{equation}
where the stages $\Phi_n^i$, $i=1, 2$, are given by
\begin{equation}
    \begin{aligned}
        \Phi_n^1 
        &= 
        Y_n + a_{11} f_\beta(t_n,Y_n) h
        + 
        \sqrt{\beta} \,
        b_1 \tfrac{\Delta Z_n}
        {h},\\
        \Phi_n^2 &= Y_n + a_{21} f_\beta(t_n,Y_n) h 
        + a_{22} f_\beta(t_n + c_1 h,\Phi_n^1) h
        + \sqrt{\beta} \,
        b_2 \tfrac{\Delta Z_n}
        {h}.
    \end{aligned}
\end{equation}
To arrive at the desired order $1.5$, coefficients of SRKSFS \eqref{eq-sfs-srk-general} must satisfy the following order conditions:
\begin{align}
    \label{eq-part-t-f}
    c_1 \omega_1 + c_2 \omega_2 &= \tfrac{1}{2}, \\
    \label{eq-D-f-f}
    a_{11} \omega_1 + (a_{21}+a_{22}) \omega_2 &= \tfrac{1}{2}, \\
    \label{eq-D-f-Z}
    b_1 \omega_1 + b_2 \omega_2 &= 1, \\
    \label{eq-D2-f}
    \tfrac{1}{2} (b_1^2 \omega_1 + b_2^2 \omega_2) &= \tfrac{3}{4},
\end{align}
where, by the It\^o Taylor expansion, conditions \eqref{eq-part-t-f}, \eqref{eq-D-f-f} and \eqref{eq-D-f-Z} are used to match the coefficients of $\partial_t f_\beta(t_n,Y_n)$, $ D f_\beta(t_n,Y_n)  f_\beta(t_n,Y_n)$ and $D f_\beta(t_n,Y_n)\tfrac{\Delta Z_{n+1}}{h}$ with those in \eqref{eq-sde-sfs-taylor}, respectively.
Condition \eqref{eq-D2-f} arises from matching $\sum_{j=1}^d D^2 f_\beta[e_j,e_j]$, using the fact that $\mathbb{E}[\|\Delta Z_{n}\|^2] = \tfrac{h^3}{3}$.
Solving these coefficient conditions, 
we obtain a simplified SRKSFS scheme requiring only two drift function evaluations:
\begin{equation}
    \label{eq-srk-appr}
    \begin{aligned}
        Y_{{n+1}} 
        &= Y_{n} 
        + \tfrac{1}{3}
        f_\beta(t_n,Y_{n})h 
        + \tfrac{2}{3}
        f_\beta(t_n+\tfrac{3}{4}h,H_{{n}})h 
        + \sqrt{\beta}
        \Delta W_n,
        \quad
        n\in [N-1]_0,
    \end{aligned}
\end{equation}
where $Y_0=0$ and the stage $H_n$ is defined by 
\begin{equation}
   H_{{n}} 
        = Y_{n} 
        + \tfrac{3}{4}f_\beta(t_n,Y_{n})h 
        + \tfrac{3\sqrt{\beta}\Delta Z_n}{2h}.  
\end{equation}
In practice, the pairs $(\Delta W_n, \Delta Z_n)_{n\in \N}$
can be generated by two sequences of independent standard Gaussian random variables \(\xi_n\), \(\eta_n \sim \mathcal{N}(0, I_d)\) via the linear transformation 
\begin{equation}\label{eq:deltaW-deltaZ-implement}
\Delta W_n 
= 
h^{\frac{1}{2}} \xi_n,
\quad 
\Delta Z_n
=
h^{\frac{3}{2}}
\big(
\tfrac{1}{2}
\xi_n 
+
\tfrac{1}{2\sqrt{3}}
\eta_n
\big).
\end{equation}

In order to attain the higher-order convergence, we need the following smoothness assumptions. 

\begin{assumption}
    \label{ass-f-3continuous}
For any $t\in[0,1]$, the drift coefficient 
$ f_\beta (t, \cdot) : \mathbb{R}^d \to \mathbb{R}^d$
of SDE (\ref{eq-temp-SDE}) 
is three times continuously differentiable 
with bounded partial derivatives:
there exists a constant $L_f > 0,$ 
which is independent of $d, t$,
such that, for any $x, v_1, v_2, v_3  \in \mathbb{R} ^d$,
    \begin{align*}
        \|Df_\beta (t,x) v_1\| \leq &L_f\|v_1\|,\\
        \|D^2f_\beta (t,x)[v_1,v_2]\| \leq & L_f\|v_1\|\cdot\|v_2\|,\\
        \|D^3f_\beta (t,x)[v_1,v_2,v_3]\| \leq & L_f\|v_1\|\cdot\|v_2\|\cdot\|v_3\|.
    \end{align*}
Moreover, for any $x\in \mathbb{R} ^d$,
the function 
$f (\cdot, x) : [ 0, 1) \to \mathbb{R} ^d$ is assumed to be twice continuously differentiable and
there exists a constant $\tilde{L} _{f}> 0$,
independent of $d$ and $t$, such that for any $t\in [ 0, 1)$, 
    \begin{align*}
    \|\partial_t f_\beta (t,x)\| &\leq \tilde{L}_f d^{\frac{1}{2}} \tfrac{1}{\sqrt{1-t}},\\
    \|\partial_{tt} f_\beta (t,x)\| &\leq \tilde{L}_f d \tfrac{1}
    {\sqrt{(1-t)^{3}}}.
    \end{align*}   
Additionally, the drift coefficient  $f(\cdot,\cdot):[0,1)\times \R^d \rightarrow \R^d$ of SDE (\ref{eq-temp-SDE}) 
is assumed to be three times continuously mixed differentiable and there exists a constant $\tilde{L}_{f}> 0$,
independent of $d$ and $t$, such that 
    \begin{align}
    \label{eq:mixed-diff}
        \|D\partial_t f_\beta (t,x)v_1\| \leq &\tilde{L}_f d^{\frac{1}{2}} \tfrac{1}{\sqrt{1-t}}\|v_1\|,
        \nonumber\\
        \big\|D^2\partial_{t} f_\beta (t,x)[v_1,v_2]\big\| \leq & \tilde{L}_f d^{\frac{1}{2}} \tfrac{1}{\sqrt{1-t}}\|v_1\|\cdot\|v_2\|.
    \end{align}
\end{assumption}
Note that Assumption \ref{ass-f-3continuous} immediately implies that, for any $x,y,v_1,v_2 \in \mathbb{R}^d,
 $ and $ t,s \in [0,1],$ 
\begin{equation}
\begin{aligned}
\label{eq:Df-lip}
\big\|
Df_{\beta}(t,x) v_1
-
Df_{\beta}(t,y) v_1
\big\|
& \leq 
L_f 
\|x-y\|\cdot\|v_1\|,\\
\big\|
D^2 f_{\beta}
(t,x)[v_1,v_2]
-
D^2 f_{\beta}
(t,y)[v_1,v_2]
\big\|
& \leq 
L_f 
\|x-y\|
\cdot\|v_1\|
\cdot\|v_2\|.
\end{aligned}
\end{equation}
Following an argument similar to that in \cite[Appendix C]{huang2025schrodinger},
we also obtain
\begin{equation}
\label{eq:Df-D2f-x-lip-t-1/2}
\begin{aligned}
\big\|
Df_{\beta}(t,x) v_1
-
Df_{\beta}(s,y) v_1
\big\|
& \leq 
\hat{L}_f 
\big(
\|x-y\|
+
d^{\frac{1}{2}}
|t-s|^{\frac{1}{2}}
\big)
\cdot\|v_1\|,\\
\big\|
D^2f_{\beta}(t,x)[v_1,v_2]
-
D^2f_{\beta}(s,y)[v_1,v_2]
\big\|
& \leq 
\hat{L}_f 
\big(
\|x-y\|
+
d^{\frac{1}{2}}
|t-s|^{\frac{1}{2}}\big)
\cdot\|v_1\|
\cdot\|v_2\|.
\end{aligned}
\end{equation}
As a direct result of \eqref{eq:mixed-diff}, we also obtain, for any $x \in \mathbb{R}^d, \  t \in [0,1)$,
\begin{equation}
    \| D\partial_t
    f_{\beta}(t, x) \|_{\mathrm{F}} 
    \leq 
    \sqrt{d} \| D\partial_tf_{\beta}(t, x) \| 
    = \sqrt{d}
    \Big( 
    \sup_{\| v_1 \|=1} \| D\partial_tf_{\beta}(t, x) v_1 \|
    \Big) 
    \leq 
    \tfrac{\tilde{L}_f d}{\sqrt{1-t}},
\end{equation}
and
\begin{equation}
\begin{aligned}
\| Df_{\beta}(t, x) 
-
Df_{\beta}(s, y) 
\|_{\mathrm{F}} 
& \leq 
\sqrt{d}
\| Df_{\beta}(t, x) 
-
Df_{\beta}(s, y) \| \\
& =
\sqrt{d}
\Big( 
\sup_{\| v_1 \|=1} 
\| Df_{\beta}(t, x) v_1
-
Df_{\beta}(s, y) v_1 \|
\Big) \\
& \leq 
\hat{L}_f
d^{\frac{1}{2}}
\big(
\|x-y\|
+
d^{\frac{1}{2}}
|t-s|^{\frac{1}{2}}
\big).
\end{aligned}
\end{equation}

It is noted that $L_f$, $\tilde{L}_f$ and $\hat{L}_f$ are generic finite constants with value 
that could change upon each appearance, 
but will not depend upon the dimension $d$ and the step size $h$.
The next proposition shows when Assumption \ref{ass-f-3continuous} is satisfied. 
The proof, postponed to Appendix \ref{app:g-condition}, follows arguments similar to those in \cite{wang2025multimodal}.
\begin{proposition}
\label{prop:g-assumption}
    For $\beta \in (0, \infty)$, 
    we assume $g_{\beta} \in \mathcal{C}^4(\mathbb{R}^d,\mathbb{R})$
    and there exist constants $L_g, \rho > 0$ 
    such that $g_{\beta}, \nabla g_{\beta}, \nabla^2 g_{\beta}, \nabla^3 g_{\beta}$ are $L_g$-Lipschitz continuous and \eqref{eq-RN-lip} holds.
    Then Assumption \ref{ass-f-3continuous} is  satisfied with 
    $L_f := \tfrac{\beta L_g}{\rho}(1+\tfrac{7L_g}{\rho}+\tfrac{12L_g^2}{\rho^2}+\tfrac{6L_g^3}{\rho^3})$ 
    and 
    $\tilde{L}_f := \tfrac{\beta^2 L_g}{2\rho}(2+\tfrac{4L_g}{\rho}+\tfrac{2L_g^2}{\rho^2})$.
\end{proposition}

Thanks to the above assumptions, we give the following higher-order non-asymptotic error bound in $L^2$-Wasserstein distance for the SRKSFS \eqref{eq-srk-appr}.

\begin{thm}
    \label{thm-exact-drift}
    (Main result: error bounds with exact drift)
    Let Assumptions \ref{ass-abso-cont}, \ref{ass-RN-lip} and \ref{ass-f-3continuous} hold. 
    Let $\{Y_n\}_{n\in [N]_0}$ be 
    defined by \eqref{eq-srk-appr} with the uniform  step size $h=1/N$.
    Then there exists a constant $C$ independent of $d,h$ and $M$, such that,
    \begin{align*}
        \mathcal{W}_2 
        \big(
        \text{Law}\,(Y_{1}), \mu
        \big) 
        \leq
        C
        (dh)^{3/2}
        |\ln h|. 
    \end{align*} 
\end{thm}
\begin{proof}
First, we recast the Schr\"odinger-F\"ollmer diffusion \eqref{eq-temp-SDE} as, for any $n\in[N]$,
\begin{equation}
    \label{eq-srk-true}
    \begin{aligned}
        X_{t_{n}} 
        &= 
        X_{t_{n-1}} 
        + 
        \tfrac{1}{3}
        f_\beta(t_{n-1},X_{t_{n-1}})h 
        + 
        \tfrac{2}{3}
        f_\beta(t_{n-1}+\tfrac{3}{4}h,H_{t_{n-1}})h 
        + 
        \sqrt{\beta}
        \Delta W_{n-1}
        +
        R_{n},  
    \end{aligned}
\end{equation}
where for short, we denote
\begin{equation}
H_{t_{n-1}} 
:= 
X_{t_{n-1}} 
+ 
\tfrac{3}{4}
f_\beta(t_{n-1},X_{t_{n-1}})h 
+ 
\tfrac{3\sqrt{\beta}
\Delta Z_{n-1}}
{2h}, 
\end{equation}
and
\begin{equation}
        R_{n} 
        := \int_{t_{n-1}}^{t_{n}} f_\beta (s,X_s)
        \,
        \mathrm{d} s 
        - 
        \tfrac{1}{3}
        f_\beta (t_{n-1},X_{t_{n-1}}) 
        h
        - 
        \tfrac{2}{3}
        f_\beta (t_{n-1}+\tfrac{3}{4}h,H_{t_{n-1}})
        h.
\end{equation}
Denoting 
$$
E_i 
:= 
X_{t_i} - Y_i,
\quad 
i \in [N]_0,
$$
and subtracting \eqref{eq-srk-true} from \eqref{eq-srk-appr} yields
\begin{equation}
\begin{aligned}
\label{eq:en-expre}
        E_n 
        & = 
        E_{n-1} 
        +
        \tfrac{1}{3}
        \big(
        f_\beta 
        (t_{n-1},X_{t_{n-1}})
        -
        f_\beta (t_{n-1},Y_{{n-1}})
        \big)h\\
        & \quad + \tfrac{2}{3}
        \big(
        f_\beta (t_{n-1}
        +\tfrac{3}{4}h,H_{t_{n-1}})
        -
        f_\beta (t_{n-1}+\tfrac{3}{4}h,H_{{n-1}})
        \big)
        h
        +
        R_{n}\\
        & =
        \tfrac{1}{3}
        \sum_{i=0}^{n-1}
        \big(
        f_\beta (t_{i},X_{t_{i}})-f_\beta (t_{i},Y_{{i}})\big)
        h\\
        & \quad + 
        \tfrac{2}{3}
        \sum_{i=0}^{n-1}
        \big(
        f_\beta (t_{i}+\tfrac{3h}{4},H_{t_{i}})
        -
        f_\beta (t_{i}+\tfrac{3h}{4},H_{{i}})
        \big)h
        +
        \sum_{i=1}^{n}
        R_{i},
\end{aligned}
\end{equation}
where we also used the fact that
$X_{0}
=
Y_{0}=0$.
Squaring both sides of \eqref{eq:en-expre}, taking expectations and using the Lipschitz conditions \eqref{eq-lip-fx},
we obtain
\begin{align}
\label{eq:error-propagation}
\mathbb{E}
\big[
\|E_n\|^2
\big]
&\leq
\tfrac{h^2}{3}
\mathbb{E}
\Big[
\big\|
\sum_{i=0}^{n-1}
\big(
f_\beta (t_{i},X_{t_{i}})
-
f_\beta (t_{i},Y_{{i}})
\big)
\big\|^2
\Big]
\nonumber\\
&\quad+ 
\tfrac{4h^2}{3}
\mathbb{E}
\Big[
\big\|
\sum_{i=0}^{n-1}
\big(
f_\beta (t_{i}+\tfrac{3h}{4},H_{t_{i}})
-
f_\beta (t_{i}+\tfrac{3h}{4},H_{{i}})
\big)
\big\|^2 
\Big]
+
3
\,
\mathbb{E}
\Big[
\big\|
\sum_{i=1}^{n}R_{i}
\big\|^2
\Big]
\nonumber\\
&\leq
\tfrac{ L_f^2 h}
{3} 
\sum_{i=0}^{n-1}
\mathbb{E}
\big[
\|E_i\|^2
\big]
+
\tfrac{4L_f^2h}{3}
\sum_{i=0}^{n-1}
\mathbb{E}
\Big[
\big\|
E_i
+
\tfrac{3}{4}
\big(
f_\beta (t_{i},X_{t_{i}})
-
f_\beta (t_{i},Y_{{i}})
\big)
\big\|^2
\Big]
+
3
\,
\mathbb{E}
\Big[
\big\|
\sum_{i=1}^{n}R_{i}
\big\|^2
\Big]
\nonumber
\\
& \leq
\big(
3+\tfrac{3L_f^2}{2}
\big)
L_f^2h
\sum_{i=0}^{n-1}
\mathbb{E}
\big[
\|E_i\|^2
\big]
+
3
\,
\mathbb{E}
\Big[
\big\|
\sum_{i=1}^{n}R_{i}
\big\|^2
\Big].
\end{align}
Before proceeding further, we first note that
\begin{equation}
\begin{aligned}
\sum_{i=1}^{n}R_{i} 
&= 
\sum_{i=1}^{n}
\bigg(
\int_{t_{i-1}}^{t_i} 
f_\beta (s,X_s)
\mathrm{d} s 
-
\tfrac{1}{3}
f_\beta 
(t_{i-1},X_{t_{i-1}})
h
- 
\tfrac{2}{3}
f_\beta 
(t_{i-1}+\tfrac{3}{4}h,H_{t_{i-1}})
h
\bigg).
\end{aligned}
\end{equation}
Then one can do the following error decomposition: 
\begin{align}
\label{eq-srk-sum-Ri}
\sum_{i=1}^{n}R_{i} 
&= 
\sum_{i=1}^{n}
\int_{t_{i-1}}^{t_i} 
\big[ 
f_\beta (s,X_s) 
- 
f_\beta (t_{i-1},X_{t_{i-1}})
\big]
\mathrm{d} s
\nonumber\\
& \quad
-
\tfrac{2}{3}
\sum_{i=1}^{n}
\Big(
f_\beta (t_{i-1}+\tfrac{3}{4}h,H_{t_{i-1}})
-f_\beta (t_{i-1},X_{t_{i-1}})
\Big)h 
\nonumber\\
&= 
\underbrace{
\sum_{i=1}^{n}
\int_{t_{i-1}}^{t_i} 
\big[ 
f_\beta (s,X_s) 
- 
f_\beta (t_{i-1},X_{t_{i-1}}) 
\big]
\mathrm{d} s
}_{=:\mathcal{H}_1}
\nonumber\\
& \quad  -
\underbrace{
\tfrac{2}{3}
\sum_{i=1}^{n}
\Big( 
f_\beta (t_{i-1}+\tfrac{3}{4}h,H_{t_{i-1}})
-
f_\beta (t_{i-1},H_{t_{i-1}}) 
\Big)h 
}_{=:\mathcal{H}_2}
\nonumber\\
& \quad  -
\underbrace{
\tfrac{2}{3}
\sum_{i=1}^{n}
\Big(                 
f_\beta (t_{i-1},H_{t_{i-1}})
-
f_\beta (t_{i-1},X_{t_{i-1}})  
\Big)h
}_{=:\mathcal{H}_3}. 
\end{align}
Applying the It\^o formula for the estimate of $\mathcal{H}_1$, for $ s \in [t_{i-1}, t_i]$ we deduce
\begin{equation}
    \label{eq-srk-H1}
    \begin{aligned}
& f_\beta (s,X_s) 
- 
f_\beta (t_{i-1},X_{t_{i-1}})  \\
        & \quad = 
        \int_{t_{i-1}}^{s}
            \Big[
            \partial_t f_\beta (r,X_r)
            +
            D f_\beta (r,X_r) 
            f_\beta (r,X_r)
                + \tfrac{\beta}
                {2}
                \sum_{j=1}^{d}
                D^2 f_\beta (r,X_{r})
                \left[ e_j,e_j \right]                    
            \Big]
            \mathrm{d} r \\
        & \quad \quad
                + \sqrt{\beta}\int_{t_{i-1}}^{s}
                D f_\beta (r,X_r)
                 \mathrm{d} W_r,
    \end{aligned}
\end{equation}
where $\{e_j\}_{j \in\{1, \cdots, d\}}$ is denoted as the orthonormal basis of $\mathbb{R}^d$.
Before proceeding, we start with the following property, which is a direct result of the Fubini theorem, for any $\varphi: [0,1]\times \R^d \rightarrow \R^d$:
\begin{equation}
\label{eq:inter-integration}
    \begin{aligned}
    \int_{t_{i-1}}^{t_i} 
    \int_{t_{i-1}}^{s} 
    \varphi (r,X_r)
    \,
    \mathrm{d} r 
    \,
    \mathrm{d} s 
    &= 
    \int_{t_{i-1}}^{t_i} 
    \int_{r}^{t_i} 
    \varphi(r,X_r)
    \,
    \mathrm{d}s 
    \,
    \mathrm{d}r\\
    & =
    \int_{t_{i-1}}^{t_i}
    (t_i-r)
    \varphi(r,X_r)
    \,
    \mathrm{d} r.
    \end{aligned}
\end{equation}
Plugging \eqref{eq-srk-H1} into the estimate of $\mathcal{H}_1$ and using \eqref{eq:inter-integration} imply
\begin{equation}
\label{eq:h1-est}
    \begin{aligned}
        \mathcal{H}_1
        &=
        \sum_{i=1}^{n}
        \int_{t_{i-1}}^{t_i}
        (t_i-r)
         \Big(\partial_t f_\beta (r,X_r)
                +
                D f_\beta (r,X_r) f_\beta (r,X_r)
            \Big) \mathrm{d} r    \\
        &\quad
                +\tfrac{{\beta}}{2}
                \sum_{i=1}^{n}
                \sum_{j=1}^{d}
                \int_{t_{i-1}}^{t_i}
                (t_i-r)
                D^2 f_\beta (r,X_{r})
                [ e_j,e_j ]
            \mathrm{d} r\\
        &\quad 
                +\sqrt{\beta}
                \sum_{i=1}^{n}                
                \int_{t_{i-1}}^{t_i}
                (t_i-r)   
                D f_\beta (r,X_r)
                 \mathrm{d} W_r.
    \end{aligned}
\end{equation}
For $\mathcal{H}_2$, recalling the Taylor expansion, we note that
\begin{equation}
    \label{eq-srk-H2}
    \begin{aligned}
& 
f_\beta (t_{i-1}+\tfrac{3}{4}h,H_{t_{i-1}})
-
f_\beta (t_{i-1},H_{t_{i-1}})\\
& \quad =   
\int_{t_{i-1}}^{t_{i-1}+\frac{3}{4}h}   
            \Big[
                \partial_t                                               
                f_\beta (t_{i-1},H_{t_{i-1}}) 
                \,
                + \int_{t_{i-1}}^{r}
                \partial_{tt} f_\beta (u,H_{t_{i-1}})
                 \,\mathrm{d} u
            \Big]
            \mathrm{d} r\\
&  \quad =
\tfrac{3}{4}
\partial_t               f_\beta (t_{i-1},H_{t_{i-1}}) 
h
+
\int_{t_{i-1}}^{t_{i-1}+\frac{3}{4}h}   
\int_{t_{i-1}}^{r}
\partial_{tt}
f_\beta (u,H_{t_{i-1}})
\,
\mathrm{d} u
\,
\mathrm{d} r,
\end{aligned}
\end{equation}
as thus 
\begin{equation}
\label{eq:h2-est}
\begin{aligned}
\mathcal{H}_2
& =
\sum_{i=1}^{n}
\int_{t_i}^{t_{i+1}}
\tfrac{h}{2}
\,
\partial_t               
f_\beta (t_{i-1},H_{t_{i-1}}) 
\dd r\\
& \quad +
\tfrac{2}{3}
\sum_{i=1}^{n}
\int_{t_i}^{t_{i+1}}
\int_{t_{i-1}}^{t_{i-1}+\frac{3}{4}h}   
\int_{t_{i-1}}^{r}
\partial_{tt} 
f_\beta (u,H_{t_{i-1}})
\,
\mathrm{d} u
\,
\mathrm{d} r
\,
\mathrm{d} s.
\end{aligned}
\end{equation}
By abuse of notation, we denote 
\begin{equation}
    \Delta X_i := H_{X_i}-X_{t_{i}} 
    = \tfrac{3}{4} f_\beta (t_{i}, X_{t_{i}})h 
    + \tfrac{3\sqrt{\beta}\Delta Z_i}{2h},
    \quad i \in [N]_0.
\end{equation}
Regarding $\mathcal{H}_3$, 
in the same manner, using the Taylor expansion yields
\begin{equation}
\label{eq-dx-f}
    \begin{aligned}
        & f_\beta (t_{i-1},H_{t_{i-1}})
-
f_\beta (t_{i-1},X_{t_{i-1}})  \\
        & \quad =
        D f_\beta (t_{i-1},X_{t_{i-1}})
        \bigg(
         \tfrac{3}{4} f_\beta (t_{i-1}, X_{t_{i-1}})h 
    + 
    \tfrac{3\sqrt{\beta}
    \Delta Z_{i-1}}{2h}
    \bigg)
        +
        \tfrac{1}{2}                
            \underbrace{     
                D^2 
                f_\beta (t_{i-1},X_{t_{i-1}})
                [\Delta X_{i-1},\Delta X_{i-1}] 
            }_{=: \mathcal{J}_1}\\
        & \quad \quad +
        \tfrac{1}{6}
        \int_{0}^{1}
                (1-u)^2
                D^3
                f_\beta (t_{i-1},X_{t_{i-1}}+u\Delta X_{i-1})
                \big[
                \Delta X_{i-1}, \Delta X_{i-1}, \Delta X_{i-1}
                \big]  \,
            \mathrm{d} u.     
    \end{aligned}
\end{equation}
For the error analysis of $\mathcal{J}_1$, we make a decomposition as follows:  
\begin{equation}
\label{eq:j1-decomposition}
    \begin{aligned}
        \mathcal{J}_1
        &= 
        D^2 f_\beta (t_{i-1},X_{t_{i-1}})
        [\Delta X_{i-1},\Delta X_{i-1}]
        -
        D^2 
        f_\beta (t_{i-1},X_{t_{i-1}})
        \big[
        \tfrac{3\sqrt{\beta}\Delta Z_{i-1}}{2h}, \tfrac{3\sqrt{\beta}\Delta Z_{i-1}}{2h}
        \big]\\
        &\quad +
        \tfrac{9{\beta}}{4h^2}
        \bigg(    
        \sum_{j=1}^{d}
        \tfrac{\partial^2 }
        {\partial x_j^2}
        f_\beta (t_{i-1},X_{t_{i-1}})
        (\Delta {Z_{i-1}^{(j)}})^2 
        +
        \sum_{j\neq k}^{d}
            \tfrac{\partial^2 }{\partial x_j\partial x_k}
            f_\beta (t_{i-1},X_{t_{i-1}})
            \Delta {Z_{i-1}^{(j)}}\Delta
            {Z_{i-1}^{(k)}}
        \bigg).
    \end{aligned}
\end{equation}
Combining \eqref{eq-dx-f} with \eqref{eq:j1-decomposition} yields
\begin{align}
\label{eq:h3-est}
\mathcal{H}_3
& =
\sum_{i=1}^{n}
\int_{t_i}^{t_{i+1}}
\tfrac{h}{2}
D f_\beta (t_{i-1},X_{t_{i-1}})
f_\beta (t_{i-1}, X_{t_{i-1}})
\dd r
+
\sqrt{\beta}
\sum_{i=1}^{n}
D f_\beta (t_{i-1},X_{t_{i-1}})
\Delta Z_{i-1}
\nonumber\\
& \quad +
\tfrac{1}{3}
\sum_{i=1}^{n}
\int_{t_i}^{t_{i+1}}
\Big[
D^2 f_\beta (t_{i-1},X_{t_{i-1}})
[\Delta X_{i-1},\Delta X_{i-1}]
-
D^2 
f_\beta (t_{i-1},X_{t_{i-1}})
\big[
\tfrac{3\sqrt{\beta}\Delta Z_{i-1}}
{2h}, \tfrac{3\sqrt{\beta}\Delta Z_{i-1}}
{2h}
\big]
\Big]
\dd r
\nonumber\\
& \quad +
\tfrac{3{\beta}}
{4h^2}
\sum_{i=1}^{n} 
\int_{t_i}^{t_{i+1}}
\sum_{j=1}^{d}
\tfrac{\partial^2 }
{\partial x_j^2}
f_\beta (t_{i-1},X_{t_{i-1}})
(\Delta {Z_{i-1}^{(j)}})^2
\dd r 
\nonumber\\
& \quad +
\tfrac{3\beta}{4h^2}
\sum_{i=1}^{n}
\int_{t_i}^{t_{i+1}}
\sum_{j\neq k}^{d}
\tfrac{\partial^2 }{\partial x_j\partial x_k}
f_\beta (t_{i-1},X_{t_{i-1}})
\Delta {Z_{i-1}^{(j)}}\Delta
{Z_{i-1}^{(k)}}
\,
\dd r
\nonumber\\
& \quad +
\tfrac{1}{9}
\sum_{i=1}^{n} \int_{t_{i-1}}^{t_i}     \int_{0}^{1}
(1-u)^2
D^3
f_\beta (t_{i-1},X_{t_{i-1}}+u\Delta X_{i-1})
[\Delta X_{i-1}, 
\Delta X_{i-1}, 
\Delta X_{i-1}]  
\,
\mathrm{d} u 
\,
\mathrm{d} r .
\end{align}
Inserting \eqref{eq:h1-est}, \eqref{eq:h2-est} and \eqref{eq:h3-est} into \eqref{eq-srk-sum-Ri} and rearranging it give
\begin{equation}
\label{eq-srk-sum-Ri-re}
    \sum_{i=1}^{n}R_{i} 
    = 
    \mathcal{I}_1+
    \mathcal{I}_2+
    \mathcal{I}_3+
    \mathcal{I}_4+
    \mathcal{I}_5+
    \mathcal{I}_6,
\end{equation}
where
\begin{align}
\mathcal{I}_1
& :=
    \sum_{i=1}^{n}
    \int_{t_{i-1}}^{t_i}
    \Big[        
    (t_i-r)
    \partial_t f_\beta (r,X_r) 
    -
    \tfrac{h}{2}
    \partial_t 
    f_\beta (t_{i-1},H_{t_{i-1}})     
    \Big] 
    \mathrm{d} r,
    \nonumber\\
\mathcal{I}_2
& :=
    \sum_{i=1}^{n}
    \int_{t_{i-1}}^{t_i}
    \Big[ 
    (t_i-r)
    D f_\beta (r,X_r)
    f_\beta (r,X_r) 
    -
    \tfrac{h}{2}
    D  f_\beta (t_{i-1},X_{t_{i-1}})
    f_\beta (t_{i-1},X_{t_{i-1}})     
    \Big]
    \mathrm{d} r,
    \nonumber\\
\mathcal{I}_3
& :=
    \tfrac{{\beta}}{2}
    \sum_{i=1}^{n}
    \sum_{j=1}^{d}
    \int_{t_{i-1}}^{t_i}
    \Big[  
    (t_i-r)
    \tfrac{\partial^2 }{\partial x_j^2}
    f_\beta (r,X_{r})
    -
    \tfrac{3(\Delta {Z_{i-1}^{(j)}})^2}{2h^2}
    \tfrac{\partial^2 }{\partial x_j^2}   
    f_\beta (t_{i-1},X_{t_{i-1}})
    \Big]
    \mathrm{d} r,
    \nonumber\\
\mathcal{I}_4
& :=
    \sqrt{\beta}
    \sum_{i=1}^{n}           \bigg( 
    \int_{t_{i-1}}^{t_i}
    (t_i-r)
    D f_\beta (r,X_r)
    \mathrm{d} W_r
    -
    D f_\beta (t_{i-1},X_{t_{i-1}}) 
    \Delta Z_{i-1}           \bigg) ,
    \nonumber\\
\mathcal{I}_5
& :=
    -
    \tfrac{2}{3}
    \sum_{i=1}^{n}
    \int_{t_{i-1}}^{t_i}
    \int_{t_{i-1}}^{t_{i-1}+\frac{3}{4}h}
    \int_{t_{i-1}}^{r}
    \partial_{tt}            f_\beta (u,H_{t_{i-1}}) 
    \dd u
    \,
    \dd r
    \,
    \dd s,
    \nonumber\\
\mathcal{I}_6
& :=
    -\tfrac{1}{3}
    \sum_{i=1}^{n} 
    \int_{t_{i-1}}^{t_i} 
    \Big[
    D^2 
    f_\beta (t_{i-1},X_{t_{i-1}})
    [\Delta X_{i-1}, \Delta X_{i-1}]
    -
    D^2 
    f_\beta (t_{i-1},X_{t_{i-1}})
    [\tfrac{3\sqrt{\beta}
    \Delta Z_{i-1}}
    {2h}, \tfrac{3\sqrt{\beta}
    \Delta Z_{i-1}}{2h}]
    \Big]
    \mathrm{d} s
    \nonumber\\
& \quad   
        -\tfrac{3{\beta}}{4h^2}
        \sum_{i=1}^{n}             
        \int_{t_{i-1}}^{t_i}                    
        \sum_{j\neq k}^{d}
        \tfrac{\partial^2 }{\partial x_j\partial x_k}
        f_\beta (t_{i-1},X_{t_{i-1}})
        \Delta {Z_{i-1}^{(j)}}\Delta {Z_{i-1}^{(k)}}
    \mathrm{d} r          
    \nonumber  \\
&\quad 
    -\tfrac{1}{9}
    \sum_{i=1}^{n} 
    \int_{t_{i-1}}^{t_i}                 
    \int_{0}^{1}
    (1-u)^2
    D^3
    f_\beta (t_{i-1},X_{t_{i-1}}+u\Delta X_i)
    [\Delta X_{i-1}, \Delta X_{i-1}, \Delta X_{i-1}]  
    \mathrm{d} u 
    \,
    \mathrm{d} r .
\end{align}
Squaring both sides of \eqref{eq-srk-sum-Ri-re}, taking expectations and using the inequality $(\sum_{i=1}^k u_i )^q \leq k^{q-1} \sum_{i=1}^k u_i^q$, $q\geq 1$, $u_i\in \R$,
we arrive at
\begin{align*}
    \mathbb{E}
    \Big[
    \big\|
    \sum_{i=1}^{n}R_{i}
    \big\|^2 
    \Big]
    \leq 
    6         
    \Big(
    \mathbb{E} 
    \big[
    \|\mathcal{I}_1\|^2
    \big]
    +
    \mathbb{E} 
    \big[
    \|\mathcal{I}_2\|^2
    \big]
    +
    \mathbb{E} 
    \big[
    \|\mathcal{I}_3\|^2
    \big]
    +
    \mathbb{E} 
    \big[
    \|\mathcal{I}_4\|^2
    \big]
    +
    \mathbb{E} 
    \big[
    \|\mathcal{I}_5\|^2
    \big]
    +
    \mathbb{E} 
    \big[
    \|\mathcal{I}_6\|^2
    \big]
    \Big).
\end{align*}
In the following, we bound these items separately.
With regard to the first term, we make a further decomposition:
\begin{equation}
\begin{aligned}
    \mathbb{E} 
    \big[
    \|\mathcal{I}_1\|^2
    \big]
    & \leq
    \underbrace{
    3\,
    \mathbb{E}
    \bigg[
    \Big\|
    \sum_{i=1}^{n}
    \int_{t_{i-1}}^{t_i}
    \big[        
    (t_i-r)
    \big(
    \partial_t f_\beta(r,X_{r}) 
    -
    \partial_t 
    f_\beta (t_{i-1},X_{t_{i-1}}) 
    \big)
    \big]
    \mathrm{d} r
    \Big\|^2
    \bigg]
    }_{=:\mathcal{I}_{1,1}}\\
    & \quad +
    \underbrace{
    3\,
    \mathbb{E}
    \bigg[
    \Big\|
    \sum_{i=1}^{n}
    \int_{t_{i-1}}^{t_i}
    \big[        
    (t_i-r)
    \partial_t f_\beta (t_{i-1},X_{t_{i-1}}) 
    -
    \tfrac{h}{2}
    \partial_t 
    f_\beta (t_{i-1},X_{t_{i-1}})     
    \big] 
    \mathrm{d} r 
    \Big\|^2
    \bigg]}_{=:\mathcal{I}_{1,2}}\\
    & \quad +
    \underbrace{
    \tfrac{3h^2}{4}
    \mathbb{E}
    \bigg[
    \Big\|
    \sum_{i=1}^{n}
    \int_{t_{i-1}}^{t_i}
    \big[        
    \partial_t f_\beta (t_{i-1},X_{t_{i-1}}) 
    -
    \partial_t 
    f_\beta (t_{i-1},H_{t_{i-1}})     
    \big] 
    \mathrm{d} r 
    \Big\|^2
    \bigg]}_{=:\mathcal{I}_{1,3}}.
\end{aligned}
\end{equation}
Again, applying the It\^o formula to $\partial_t f_{\beta}(t,x)$ shows
\begin{equation}
\begin{aligned}
& 
\partial_t f_{\beta}(r,X_r)
-
\partial_t f_{\beta}(t_{i-1},X_{t_{i-1}})\\
& = 
\int_{t_{i-1}}^{r}
\Big[
\partial_{tt} f_\beta (u,X_u)
+
D \partial_t 
f_{\beta} (u,X_u) 
f_\beta (u,X_u)
+ 
\tfrac{\beta}
{2}
\sum_{j=1}^{d}
D^2 \partial_t f_{\beta}(u,X_{u})
\left[ e_j,e_j \right]   \Big]
\mathrm{d} u \\
& \quad \quad
+ 
\sqrt{\beta}
\int_{t_{i-1}}^{r}
D \partial_t f_{\beta} (u,X_u)
\mathrm{d} W_u.
\end{aligned}
\end{equation}
Inserting this into the estimate of 
$\mathcal{I}_{1,1}$ results in
\begin{align}
\label{eq:I11-decompositon}
    \mathcal{I}_{1,1}
    & \leq
    9\,
    \mathbb{E}
    \bigg[
    \Big\|
    \sum_{i=1}^{n}
    \int_{t_{i-1}}^{t_i}
    \int_{t_{i-1}}^{r}
    (t_i-r)
    \partial_{tt}
    f_\beta (u,X_u) 
    \,
    \mathrm{d}u 
    \,
    \mathrm{d}r 
    \Big\|^2
    \bigg]
    \nonumber\\
    & \quad +
    9\,
    \mathbb{E}
    \bigg[
    \Big\|
    \sum_{i=1}^{n}
    \int_{t_{i-1}}^{t_i}
    \int_{t_{i-1}}^{r}
    \big[
     (t_i-r)
     \big(
     D \partial_{t} f_\beta (u,X_u)
     f_\beta (u,X_u) 
     \nonumber\\
    & \qquad \qquad 
    +
    \tfrac{\beta}
    {2}
    \sum_{j=1}^{d}
    D^2
    \partial_{t}
    f_\beta (u,X_u)
    [e_j,e_j]
    \big) 
    \big]
    \mathrm{d}u 
    \,
    \mathrm{d}r 
    \Big\|^2
    \bigg]\nonumber\\
    & \quad +
    9\,
    \mathbb{E}
    \bigg[
    \Big\|
    \sum_{i=1}^{n}
    \int_{t_{i-1}}^{t_i}
    \int_{t_{i-1}}^{r}
    (t_i-r)
    D \partial_{t} f_\beta (u,X_u)
    \mathrm{d} W_u 
    \,
    \mathrm{d}r 
    \Big\|^2
    \bigg]
    \nonumber\\
    & =:
\mathcal{I}^{(1)}_{1,1}
    +
\mathcal{I}^{(2)}_{1,1}
    +
\mathcal{I}^{(3)}_{1,1}.
\end{align}

\noindent
Using \eqref{eq:inter-integration}, the Minkowski inequality, the H\"older inequality and Assumption \ref{ass-f-3continuous}, 
we treat $\mathcal{I}^{(1)}_{1,1}$ as follows:
\begin{equation}
\begin{aligned}
\mathcal{I}^{(1)}_{1,1}
& =
9
\,
\Big\|
\sum_{i=1}^{n}
\int_{t_{i-1}}^{t_i}
\int_{u}^{t_i}
(t_i-r)
\cdot 
\partial_{tt}
f_\beta (u,X_u)
\mathrm{d}r 
\,
\mathrm{d}u
\Big\|^2_{L^2(\Omega;\mathbb{R}^d)} \\
& \leq 
\tfrac{9}{4}
\bigg(
\sum_{i=1}^{N}
\int_{t_{i-1}}^{t_i}
(t_i-u)^2
\big\|
\partial_{tt}
f_\beta (u,X_u)
\big\|_{L^2(\Omega;\mathbb{R}^d)}
\,
\mathrm{d} u             
\bigg)^2\\
&\leq
\tfrac{9}{4}
\bigg(
\tilde{L}_f 
d h^{3/2}
\sum_{i=1}^{N-1}
\int_{t_{i-1}}^{t_i} 
\tfrac{1}{1-u}  
\mathrm{d} u 
+
\tilde{L}_f 
d
h
\int_{1-h}^{1} 
\tfrac{1}{\sqrt{1-u}}  
\mathrm{d} u
\bigg)^2\\
&\leq
C d^2 h^3 |\ln h|^2,
\end{aligned}
\end{equation}
where it is straightforward to show
\[
\sum_{i=1}^{N-1}
\int_{t_{i-1}}^{t_i} 
\tfrac{1}{1-u}  
\mathrm{d} u 
=
\int_h^1
\tfrac{1}{s}
\,
\mathrm{d} s
=
| \ln h|, 
\quad \quad 
\int_{1-h}^{1} 
\tfrac{1}{\sqrt{1-u}}  
\mathrm{d} u
=
\int_0^h
\tfrac{1}{\sqrt{s}}
\,
\mathrm{d} s
=
2 h^{\frac{1}{2}}.
\]

\noindent
Regarding $\mathcal{I}^{(2)}_{1,2}$, by the H\"older inequality and Assumption \ref{ass-f-3continuous}, we deduce
\begin{equation}
\begin{aligned}
    \label{eq-srk-I12}   
\mathcal{I}^{(2)}_{1,2}
&\leq
     18
     \,
     \mathbb{E}
     \bigg[
     \Big\|
     \sum_{i=1}^{n}
     \int_{t_{i-1}}^{t_i}
     \int_{t_{i-1}}^{r}
     (t_i-r)
    D \partial_{t} f_\beta (u,X_u)
    f_\beta (u,X_u) 
    \mathrm{d}u 
    \,
    \mathrm{d}r 
    \Big\|^2 
    \bigg]\\
    &\quad
            + 
    \tfrac{9 \beta^2}{2}
    \mathbb{E}
    \bigg[
    \Big\|
    \sum_{i=1}^{n}
    \int_{t_{i-1}}^{t_i}
    \int_{t_{i-1}}^{r}
   (t_i-r)
    \sum_{j=1}^{d}
    D^2\partial_{t}
    f_\beta (u,X_u)
    [e_j,e_j]
    \mathrm{d}u
    \,
    \mathrm{d}r 
    \Big\|^2 
    \bigg]\\
&\leq 
18h
\sum_{i=1}^{n}
\int_{t_{i-1}}^{t_i} 
\int_{t_{i-1}}^{r}    
(t_i-r)^2
\,
\mathbb{E}
\Big[
\big\|
D \partial_t
f_\beta (u,X_{u})
f_\beta (u,X_u)
\big\|^2 
            \Big]
            \mathrm{d}u 
            \,
            \mathrm{d}r
\\
&\quad 
            + \tfrac{9\beta^2dh}{2}
       \sum_{i=1}^{n}
       \sum_{j=1}^{d}
            \int_{t_{i-1}}^{t_i} 
            \int_{t_{i-1}}^{r}
            (t_i-r)^2
            \,
            \mathbb{E}    \Big[      
            \big\|
            D^2 \partial_t
            f_\beta (r,X_{r})
            [e_j,e_j]
            \big\|^2
            \Big]
            \mathrm{d}u
            \,
            \mathrm{d}r     
\\
& \leq 
\big(
4(\hat{L}_{f}^2 d+L_f^2 M_1 d)+ \beta^2 d^2
\big)
\tfrac{9 \tilde{L}^2_f
d h } {2}
\sum_{i=1}^{n}
\int_{t_{i-1}}^{t_i} 
\int_{t_{i-1}}^{r}
\tfrac{(t_i-r)^2}{1-u}
\mathrm{d}u
\mathrm{d}r \\
& \leq 
C d^3 h^3,
\end{aligned}
\end{equation}
where we used the fact
\begin{equation*}
\sum_{i=1}^{n}
\int_{t_{i-1}}^{t_i} 
\int_{t_{i-1}}^{r}
\tfrac{(t_i-r)^2}{1-u}
\mathrm{d}u
\,
\mathrm{d}r
\leq 
\sum_{i=1}^{N}
\int_{t_{i-1}}^{t_i}
\tfrac{(t_i-u)^3}{3({1-u})}
\mathrm{d}u
\leq 
h^2.
\end{equation*}
Next we cope with $\mathcal{I}^{(3)}_{1,1}$. First, we have 
\begin{equation}
\begin{aligned}
\mathcal{I}^{(3)}_{1,1}
& =
\underbrace{
9\,
\sum_{i=1}^{n}
    \mathbb{E}
    \bigg[
    \Big\|
    \int_{t_{i-1}}^{t_i}
    \int_{t_{i-1}}^{r}
    (t_i-r)
    D \partial_{t} f_\beta (u,X_u)
    \mathrm{d} W_u 
    \,
    \mathrm{d}r 
    \Big\|^2
    \bigg]}_{=:T_1}\\
& \quad +
18 \,
\sum_{1 \leq i < j \leq n}
\underbrace{
\mathbb{E}
\bigg[
\Big \langle
    \int_{t_{i-1}}^{t_i}
    \int_{t_{i-1}}^{r}
    (t_i-r)
    D \partial_{t} f_\beta (u,X_u)
    \mathrm{d} W_u 
    \,
    \mathrm{d}r 
,
\int_{t_{j-1}}^{t_j}
    \int_{t_{j-1}}^{r}
    (t_j-r)
    D \partial_{t} f_\beta (u,X_u)
    \mathrm{d} W_u 
    \,
    \mathrm{d}r
    \Big \rangle
\bigg]}_{=:T_2}.
\end{aligned}
\end{equation}
For any $1\leq i < j \leq n$, we show that the second term vanishes: 
\begin{equation}
\begin{aligned}
T_2 
& =
\E
\bigg[
\E
\bigg[
\Big \langle
    \int_{t_{i-1}}^{t_i}
    \int_{t_{i-1}}^{r}
    (t_i-r)
    D \partial_{t} f_\beta (u,X_u)
    \mathrm{d} W_u 
    \,
    \mathrm{d}r 
,
\int_{t_{j-1}}^{t_j}
    \int_{t_{j-1}}^{r}
    (t_j-r)
    D \partial_{t} f_\beta (u,X_u)
    \mathrm{d} W_u 
    \,
    \mathrm{d}r
\bigg|
\mathcal{F}^W_{t_{j-1}}
\bigg]
\bigg]\\
& \leq
\E
\bigg[
\Big \langle
    \int_{t_{i-1}}^{t_i}
    \int_{t_{i-1}}^{r}
    (t_i-r)
    D \partial_{t} f_\beta (u,X_u)
    \mathrm{d} W_u 
    \,
    \mathrm{d}r 
,
\int_{t_{j-1}}^{t_j}
\mathbb{E}
\bigg[
\int_{t_{j-1}}^{r}
(t_j-r)
    D \partial_{t} f_\beta (u,X_u)
\mathrm{\,d} W_u
\Big|
\mathcal{F}^W_{t_{j-1}}
\bigg]
\mathrm{\,d} r
\Big \rangle
\bigg]\\
& =0,
\end{aligned}
\end{equation}
where we used the fact that the first integral is $\mathcal{F}^W_{t_{j-1}}$-measurable and applied the basic property of the It\^o integral. 
As a result, we apply  the H\"older inequality, the It\^o Isometry \cite[Lemma 5.4]{mao2007stochastic}, Assumption \ref{ass-f-3continuous} and Lemma \ref{lem:moments-bounded} to show

\begin{equation}
    \begin{aligned}
        \label{eq-srk-I13}
\mathcal{I}^{(3)}_{1,1}
=
T_1
& \leq
9 h
\sum_{i=1}^{n}
\mathbb{E}
\bigg[
\int_{t_{i-1}}^{t_i} 
\Big\|
\int_{t_{i-1}}^{r}
(t_i-r)
D \partial_t
f_\beta (u,X_{u})  
\,
\mathrm{d} W_u
\Big\|^2
\mathrm{d} r 
\bigg]\\
&=
9 h
\sum_{i=1}^{n}
\mathbb{E}
\bigg[
\int_{t_{i-1}}^{t_i} 
\int_{t_{i-1}}^{r}
(t_i-r)^2 
\big\|
D \partial_t
f_\beta (u,X_{u})  
\big\|^2_{\mathrm{F}}
\,
\mathrm{d}u
\,
\mathrm{d} r 
\bigg]\\
&\leq
9 \tilde{L}_f^2 d^2 h
\sum_{i=1}^{n}
\int_{t_{i-1}}^{t_i} 
\int_{t_{i-1}}^{r}
\tfrac{(t_i-r)^2}{1-u}
\mathrm{d} u
\,
\mathrm{d} r\\
& \leq
C d^2 h^{3}.
\end{aligned}
\end{equation}
Combining these estimates with \eqref{eq:I11-decompositon} gives
\begin{equation}
\mathcal{I}_{1,1}
    \leq
    C d^3 h^3 |\ln h|^2.
\end{equation}
Noting that
\[
\int_{t_{i-1}}^{t_i} \big((t_i-r)-\tfrac{h}{2}\big)\mathrm{d} r=0,
\]
we can easily get $\mathcal{I}_{1,2}=0.$
By the Minkowski inequality, the Taylor expansion, the H\"older inequality and Assumption \ref{ass-f-3continuous}, we estimate $\mathcal{I}_{1,3}$ as follows:
\begin{equation}
\begin{aligned}
\mathcal{I}_{1,3}
& =
\tfrac{3h^2}{4}
    \Big\|
    \sum_{i=1}^{n}
    \int_{t_{i-1}}^{t_i}
    \int_0^1
    D  \partial_t f_\beta (t_{i-1},X_{t_{i-1}}+u \Delta X_{i-1}) 
    \Delta X_{i-1}
    \mathrm{d} u
    \,
    \mathrm{d} r 
    \Big\|^2
    _{L^2(\Omega;\mathbb{R}^d)}\\
    & \leq
    \tfrac{3h^2}{4}
    \bigg(
    \sum_{i=1}^{n}
    \int_{t_{i-1}}^{t_i}
    \int_0^1
    \big\|
    D  \partial_t f_\beta (t_{i-1},X_{t_{i-1}}+u \Delta X_{i-1}) 
    \Delta X_{i-1}
    \big\|
    _{L^2(\Omega;\mathbb{R}^d)}
    \mathrm{d} u
    \,
    \mathrm{d} r 
    \bigg)^2\\
    & \leq
    \tfrac{3 h^2}
    {4}
    \bigg(
    \tilde{L}_f d^{1/2}
    \sum_{i=1}^{n}
    \int_{t_{i-1}}^{t_i}
    \tfrac{1}{\sqrt{1-t_{i-1}}} 
    \mathrm{d} r 
    \sup_{i\in [N]}
    \big\{\|
    \Delta X_{i-1}
    \|_{L^2(\Omega;\mathbb{R}^d)}
    \big\}
    \bigg)^2\\
    & \leq
    C d^2 h^3,
\end{aligned}
\end{equation}
where the fact was used that
\begin{equation}
    \begin{aligned}
        \|\Delta X_{i-1}\|_{L^2(\Omega;\mathbb{R}^d)}
    & =
        \Big\|
        \tfrac{3}{4} f_\beta (t_{i}, X_{t_{i}})h 
            + \tfrac{3\sqrt{\beta}\Delta Z_i}{2h}
        \Big\|
        _{L^2(\Omega;\mathbb{R}^d)}\\
    & \leq 
        \tfrac{3h}{4}
        \big\| 
        f_\beta (t_{i}, X_{t_{i}})
        \big\|
        _{L^2(\Omega;\mathbb{R}^d)}
        +
        \tfrac{3\sqrt{\beta}}{2h}
        \bigg\|
        \int_{t_{i-1}}^{t_i}
        ({t_{i}}-r)
        \dd W_r
        \bigg\|_{L^2(\Omega;\mathbb{R}^d)} \\
    &\leq
    \Big(
    \tfrac{3}{4}
    \tilde{L}_f
    (1+\sqrt{M_1})
    +
    \tfrac{\sqrt{3\beta}}{4}
    \Big)
    d^{\frac{1}{2}} h^{\frac{1}{2}}.
    \end{aligned}
\end{equation}
Equipped with the above estimates,  we conclude 
\begin{equation}
    \mathbb{E} 
    \big[
    \|\mathcal{I}_1\|^2
    \big]
    \leq 
    C d^3 h^3
    |\ln h|^2.
\end{equation}
In order to handle $\mathcal{I}_2$,  
we begin with the following decomposition:
\begin{equation}
\label{eq:Dff-decom}
    \begin{aligned}
                D f_\beta (r,X_r) 
                f_\beta (r,X_r)
            &= 
            Df_\beta (t_{i-1},X_{t_{i-1}})
            f_\beta (t_{i-1},X_{t_{i-1}}) 
            \\
            &\quad 
            +
            \big(
                D f_\beta (r,X_r)-Df_\beta (t_{i-1},X_{t_{i-1}}) 
            \big)
                f_\beta (r,X_r)
             \\
            &\quad 
                +Df_\beta (t_{i-1},X_{t_{i-1}})
            \big(
                f_\beta (r,X_r)
                -
                f_\beta (t_{i-1},X_{t_{i-1}})
            \big).
    \end{aligned}
\end{equation}
In view of \eqref{eq:Dff-decom}, we obtain 
\begin{equation}
\begin{aligned}
\mathbb{E} 
\big[
\|\mathcal{I}_2\|^2
\big]
 &\leq
 \underbrace{
 3
 \,
 \mathbb{E}
 \bigg[
 \Big\|
\sum_{i=1}^{n}
\int_{t_{i-1}}^{t_i}
\big(
(t_i-r) 
 -
 \tfrac{h}{2}
 \big)
        D 
        f_\beta (t_{i-1},X_{t_{i-1}})
        f_\beta (t_{i-1},X_{t_{i-1}})
        \mathrm{d} r
        \Big\|^2
\bigg]}_{=:\mathcal{I}_{2,1}}\\
        &\quad +
         \underbrace{
         3\mathbb{E} 
        \bigg[
 \Big\|
 \sum_{i=1}^{n}
        \int_{t_{i-1}}^{t_i}
        (t_i-r)
        \big(
            Df_\beta (r,X_r)
            -Df_\beta (t_{i-1},X_{t_{i-1}})
        \big)
            f_\beta (t_{i-1},X_{t_{i-1}})
            \mathrm{d} r
            \Big\|^2
            \bigg] }_{=:\mathcal{I}_{2,2}}\\
        &\quad +
    \underbrace{
        3 \mathbb{E}
        \bigg[
 \Big\|
 \sum_{i=1}^{n}
                \int_{t_{i-1}}^{t_i}
                (t_i-r)
                Df_\beta (t_{i-1},X_{t_{i-1}})
            \big(
                f_\beta (r,X_r)
                -
                f_\beta (t_{i-1},X_{t_{i-1}})
            \big)
            \mathrm{d} r
            \Big\|^2
            \bigg]}_{=:\mathcal{I}_{2,3}}. 
\end{aligned}
\end{equation}
In the following we cope with the above three items separately. 
Noting that
\[
\int_{t_{i-1}}^{t_i} \big((t_i-r)-\tfrac{h}{2}\big)\mathrm{d} r=0,
\]
one can easily see $\mathcal{I}_{2,1}=0.$
With regard to $\mathcal{I}_{2,2}$, we utilize the inequality $(\sum_{i=1}^k u_i )^q \leq k^{q-1} \sum_{i=1}^k u_i^q$, $q\geq 1$, $u_i\in \R$, the H\"older inequality, \eqref{eq:Df-D2f-x-lip-t-1/2} and Lemma \ref{lem:moments-bounded} to obtain
\begin{equation}
\label{eq:I22-est}
\begin{aligned}
\mathcal{I}_{2,2}
& \leq
3
h
\sum_{i=1}^{n}
\mathbb{E} 
\bigg[
\Big\|
\int_{t_{i-1}}^{t_i} 
\big(
Df_\beta (r,X_r)
-Df_\beta (t_{i-1},X_{t_{i-1}})
\big)
f_\beta (t_{i-1},X_{t_{i-1}})
\mathrm{d} r
\Big\|^2
\bigg] \\
& \leq
6
\hat{L}^2_f
h^2
\sum_{i=1}^{n}
\mathbb{E}
\bigg[
\int_{t_{i-1}}^{t_i}
\big(
\|X_r-X_{t_{i-1}}\|^2
+
d |r-t_{i-1}|
\big)
\cdot
\|f_\beta (t_{i-1},X_{t_{i-1}})
\|^2
\mathrm{d} r
\bigg]\\
& \leq 
6
\hat{L}^2_f
h^2
\sum_{i=1}^{n}
\int_{t_{i-1}}^{t_i}
\|X_r-X_{t_{i-1}}\|^2_{L^4(\Omega;\mathbb{R}^d)}
\cdot
\|f_\beta (t_{i-1},X_{t_{i-1}})
\|^2_{L^4(\Omega;\mathbb{R}^d)}
\mathrm{d} r
\bigg]\\
& \quad +
6
\hat{L}^2_f
d h^2
\sum_{i=1}^{n}
\int_{t_{i-1}}^{t_i}
|r-t_{i-1}|
\cdot
\|f_\beta (t_{i-1},X_{t_{i-1}})
\|^2_{L^2(\Omega;\mathbb{R}^d)}
\mathrm{d} r
\bigg]\\
& \leq
C d^2 h^3.
\end{aligned}
\end{equation}
For $\mathcal{I}_{2,3}$, in the same manner, using the inequality $(\sum_{i=1}^k u_i )^q \leq k^{q-1} \sum_{i=1}^k u_i^q$, $q\geq 1$, $u_i\in \R$, the H\"older inequality, Assumption \ref{ass-f-3continuous}, the Lipschitz condition in $x$ and $\tfrac{1}{2}$-H\"older continuous in $t$ of $f_{\beta}(t,x)$ \eqref{eq:b-lipschitz-in-x-houder-in-t} and Lemma \ref{lem:moments-bounded} to arrive at 
\begin{equation}
\begin{aligned}
\mathcal{I}_{2,3}
& \leq
    3 h
    \sum_{i=1}^{n}
    \mathbb{E}
    \bigg[
    \Big\|
    \int_{t_{i-1}}^{t_i}
    Df_\beta (t_{i-1},X_{t_{i-1}})
    \big[
    f_\beta (r,X_r)
    -
    f_\beta (t_{i-1},X_{t_{i-1}})
    \big]
    \mathrm{d} r
    \Big\|^2
    \bigg]\\
& \leq
    6 
    L^2_f\hat{L}^2_f 
    h^2
    \sum_{i=1}^{n}
    \mathbb{E}
    \bigg[
    \int_{t_{i-1}}^{t_i}
    \big(
    \|
    X_r-X_{t_{i-1}}\|^2
    +
    d
    \,
    |r-t_{i-1}|
    \big)
    \mathrm{d} r
    \bigg]\\
& \leq
    C d h^3.
\end{aligned}
\end{equation}
Gathering the above estimates, we have
\begin{equation}
\mathcal{I}_2 
\leq
C d^2 h^3.
\end{equation}

\noindent
Before treating $\mathcal{I}_3$, we first note
\begin{equation}
\label{eq:h4}
    \mathbb{E}
    \Big[
    \big\|
    \tfrac{h^2}{2} 
    -
    \tfrac{3(\Delta {Z_{i-1}^{(j)}})^2}{2h} \|^2
    \Big]
    =
    \tfrac{h^4}{2}, 
\end{equation}
and make the following decomposition:
\begin{equation}
\begin{aligned}
\mathbb{E} 
\big[
\|\mathcal{I}_3\|^2
\big]
& \leq
\underbrace{
\tfrac{\beta}{2}
\mathbb{E}
\bigg[
\Big\|
\sum_{j=1}^{d}
\sum_{i=1}^{n}
\int_{t_{i-1}}^{t_i}
(t_i-r)
\big(
D^2 f_{\beta} 
(r,X_r)[e_j,e_j]
-
D^2 f_{\beta}
(t_{i-1},X_{t_{i-1}})
[e_j,e_j]   
\big)
\mathrm{d} r
\Big\|^2
\bigg]}_{=:\mathcal{I}_{3,1}}\\
& \quad +
\underbrace{
\tfrac{\beta}{2}
\mathbb{E}
\bigg[
\Big\|
\sum_{j=1}^{d}
\sum_{i=1}^{n}
\int_{t_{i-1}}^{t_i}     \Big(
(t_i-r)
-
\tfrac{3(\Delta {Z_{i-1}^{(j)}})^2}
{2h^2}
\Big)
D^2 f_{\beta}
(t_{i-1},X_{t_{i-1}})
[e_j,e_j]
\mathrm{d} r
\Big\|^2  
\bigg]}_{=:\mathcal{I}_{3,2}}.
\end{aligned}
\end{equation}
Using the same technique as before in treating $\mathcal{I}_{2,2}$, and applying the H\"older inequality, \eqref{eq:Df-D2f-x-lip-t-1/2} and Lemma \ref{lem:moments-bounded}, we directly get 
\[
\mathcal{I}_{3,1} \leq C d^2 h^3 .
\]
Noting that 
\[
\tfrac{3}{2h}
\mathbb{E} 
\big[ 
(\Delta {Z_{i-1}^{(j)}})^2
\big]
=
\tfrac{3}{2h}
\int_{t_{i-1}}^{t_i}
(t_i-r)^2
\dd r
=
\tfrac{h^2}{2},
\]
following the same arguments as used in the estimate of $\mathcal{I}^{(3)}_{1,1}$, and using the basic property of conditional expectation, we easily show that  the cross term vanishes.
As a result, we employ
Assumption \ref{ass-f-3continuous} and \eqref{eq:h4} to derive
\begin{equation}
\begin{aligned}
\mathcal{I}_{3,2} 
& =
\tfrac{\beta}{2}
\sum_{i=1}^{n}
\mathbb{E}
\bigg[
\Big\|
\sum_{j=1}^{d}
\int_{t_{i-1}}^{t_i}     \Big(
(t_i-r)
-
\tfrac{3(\Delta {Z_{i-1}^{(j)}})^2}
{2h^2}
\Big)
D^2 f_{\beta}
(t_{i-1},X_{t_{i-1}})
[e_j,e_j]
\mathrm{d} r
\Big\|^2  
\bigg]\\
& \leq
\tfrac{\beta d}{2}
\sum_{i=1}^{n}
\sum_{j=1}^{d}
\mathbb{E}
\Big[
\|
D^2 f_{\beta}
(t_{i-1},X_{t_{i-1}})
[e_j,e_j]
\|^2 
\,
\E 
\big[
\big\|
\tfrac{h^2}{2} -\tfrac{3(\Delta {Z_{i-1}^{(j)}})^2}{2h} 
\big\|^2
\big|
\mathcal{F}_{t_{i-1}}
\big]
\Big]\\
& \leq
 C d^2 h^3.
\end{aligned}
\end{equation}
By repeating a similar argument in treating $T_2$ and using Assumption \ref{ass-f-3continuous} and the H\"older inequality, it is straightforward to show that
\begin{equation}
    \begin{aligned}
\mathbb{E} 
    \big[
    \|\mathcal{I}_4\|^2
    \big]
    & = 
                \beta\,
    \E
    \bigg[
            \Big\|
                \sum_{i=1}^{n}       
                \int_{t_{i-1}}^{t_i}
                \big[(t_i-r)
            \big(
                D f_\beta (r,X_r)                
                -D 
                f_\beta (t_{i-1},X_{t_{i-1}})          
            \big)
                \big]
                \mathrm{d} W_r
                \Big\|^2
                \bigg]\\
        & =
        \beta \,
        \sum_{i=1}^{n}   
        \E
    \bigg[
            \Big\|
                \int_{t_{i-1}}^{t_i}
                \big[(t_i-r)
            \big(
                D f_\beta (r,X_r)                
                -D 
                f_\beta (t_{i-1},X_{t_{i-1}})          
            \big)
                \big]
                \mathrm{d} W_r
                \Big\|^2
                \bigg]\\
           & \leq
           \beta h^2 
           \sum_{i=1}^{n}
           \E
           \bigg[
           \int_{t_{i-1}}^{t_i}
           \big\|
        D f_\beta (r,X_r)
        -D f_\beta (t_{i-1},X_{t_{i-1}})            
            \big\|^2_{\mathrm{F}}
            \,
            \mathrm{d} r
            \bigg]\\
        &\leq
        C d^2 h^3
    \end{aligned}
\end{equation}
and 
\begin{equation}
\begin{aligned}
\mathbb{E} 
\big[
\|\mathcal{I}_5\|^2
\big]
& =
\tfrac{2}{3}
\,
\Big\|
\sum_{i=1}^{n}
\int_{t_{i-1}}^{t_i}
\int_{t_{i-1}}^{t_{i-1}+\frac{3}{4}h}
\int_{t_{i-1}}^{r}
\partial_{tt}            f_\beta (u,H_{t_{i-1}})
\,
\dd u
\,
\dd r
\,
\dd s
\Big\|^2_{L^2(\Omega;\mathbb{R}^d)}\\
& \leq
\tfrac{2}{3}
\bigg(
\sum_{i=1}^{N}
\int_{t_{i-1}}^{t_i}
\int_{t_{i-1}}^{t_{i-1}+\frac{3}{4}h}
\int_{t_{i-1}}^{r}
\big\|
\partial_{tt}            f_\beta (u,H_{t_{i-1}})
\big\|^2_{L^2(\Omega;\mathbb{R}^d)}
\dd u
\,
\dd r
\,
\dd s
\bigg)^2 \\
& \leq
C d h^3 |\ln h|^2.
\end{aligned}
\end{equation}
Now let us start to bound $\mathcal{I}_6$:
%
\begin{equation}
\begin{aligned}
& 
\mathbb{E} 
\big[
\|\mathcal{I}_6\|^2
\big]\\
& \leq
\tfrac{1}{3}
\,
\E
\bigg[
\Big\|
\sum_{i=1}^{n} 
\int_{t_{i-1}}^{t_i} 
\Big[
D^2 
f_\beta (t_{i-1},X_{t_{i-1}})
[\Delta X_{i-1}, \Delta X_{i-1}]
-
D^2 
f_\beta (t_{i-1},X_{t_{i-1}})
[\tfrac{3\sqrt{\beta}
\Delta Z_{i-1}}
{2h}, \tfrac{3\sqrt{\beta}
\Delta Z_{i-1}}{2h}]
\Big]
\mathrm{d} s
\Big\|^2
\bigg]\\
& \quad +
\tfrac{27 \beta^2}
{16 h^4}
\,
\E
\bigg[
\Big\|
\sum_{i=1}^{n}             
        \int_{t_{i-1}}^{t_i}                    
        \sum_{j\neq k}^{d}
        \tfrac{\partial^2 }{\partial x_j\partial x_k}
        f_\beta (t_{i-1},X_{t_{i-1}})
        \Delta {Z_{i-1}^{(j)}}\Delta {Z_{i-1}^{(k)}}
    \mathrm{d} s 
    \Big\|^2
\bigg]\\
& \quad +
\tfrac{1}{27}
\,
\E
\bigg[
\Big\|
\sum_{i=1}^{n} 
\int_{t_{i-1}}^{t_i}     
\int_{0}^{1}
(1-u)^2
D^3
f_\beta (t_{i-1},X_{t_{i-1}}+u\Delta X_{i-1})
[\Delta X_{i-1}, 
\Delta X_{i-1}, 
\Delta X_{i-1}]  
\,
\mathrm{d} u 
\,
\mathrm{d} r 
\Big\|^2
\bigg]\\
& =:
\mathcal{I}_{6,1}
+
\mathcal{I}_{6,2}
+
\mathcal{I}_{6,3}.
\end{aligned}
\end{equation}
Noting that, for any matrix $ U \in \R^{d \times d}$ and any  $a, b \in \R^d$,
$$
a^T U a-b^T U b=(a-b)^T U(a-b)+(a-b)^T U b+b^T U(a-b),
$$ 
and applying the H\"older inequality, \eqref{eq-lip-fx} and Lemma \ref{lem:moments-bounded},
one has a further decomposition of the first two terms $\mathcal{I}_{6,1}$ as follows:
\begin{equation}
\label{eq:I61-est}
\begin{aligned}
\mathcal{I}_{6,1}
& \leq
\tfrac{1}{3}
\,
    \E 
    \Big[
    \big\|
    D^2 
    f_\beta (t_{i-1},X_{t_{i-1}})
    [\Delta X_{i-1}, \Delta X_{i-1}] 
    -
    D^2 
    f_\beta (t_{i-1},X_{t_{i-1}})
    [\tfrac{3\sqrt{\beta} \Delta Z_{i-1}}
    {2h}, \tfrac{3\sqrt{\beta} \Delta Z_{i-1}}
    {2h}]
    \big\|^2
    \Big]\\
& =
\tfrac{81 h^4}{256}
    \E
    \Big[
    \big\|
    D^2 f_\beta (t_{i-1},X_{t_{i-1}})
    \big[
    f_\beta(t_{i-1},X_{t_{i-1}}),
    f_\beta(t_{i-1},X_{t_{i-1}})
    \big]
    \big\|^2
    \Big]\\
& \quad 
    +
    \tfrac{9\beta}{4}
    \E
    \Big[
    \big\|
    D^2 f_\beta (t_{i-1},X_{t_{i-1}})
    \big[
    \Delta Z_{i-1},f_\beta(t_{i-1},X_{t_{i-1}})
    \big]
    \big\|^2
    \Big]\\
& \quad
    +
    \tfrac{9\beta}{4}
    \E
    \Big[
    \big\|
    D^2 f_\beta (t_{i-1},X_{t_{i-1}})
    \big[
    f_\beta(t_{i-1},X_{t_{i-1}}),
    \Delta Z_{i-1}
    \big]
    \big\|^2
    \Big]\\
& \leq  
\tfrac{81 L_f^2 h^4}{256}
    \E \big[
    \|f_\beta(t_{i-1},X_{t_{i-1}})\|^4
    \big]
    + 
    \tfrac{9\beta L_f^2}{2}
    \E\big[
    \|f_\beta(t_{i-1},X_{t_{i-1}})\|^2
    \|\Delta Z_{i-1}\|^2
    \big]\\
& \leq 
    C d^2 h^3.
\end{aligned}
\end{equation}
Let us now proceed with a similar treatment for 
$\mathcal{I}_{6,2}$. Applying the H\"older inequality and the It\^o isometry, we get
\begin{equation}
\label{eq:I62-est}
    \begin{aligned}
\mathcal{I}_{6,2} 
& \leq
\tfrac{27 \beta^2}{16 h^4}
\sum_{i=1}^{n}
\E
\bigg[ 
\Big\|
\int_{t_{i-1}}^{t_i}      \sum_{j\neq k}^{d}
\tfrac{\partial^2 }{\partial x_j\partial x_k}
        f_\beta (t_{i-1},X_{t_{i-1}})
        \Delta {Z_{i-1}^{(j)}}\Delta {Z_{i-1}^{(k)}}
\mathrm{d} s
\Big\|^2
\bigg]\\
& =
\tfrac{27 \beta^2 }{16 h^2}
\sum_{i=1}^{n}
2
\sum_{j\neq k}^d
\E
\Big[
\big\|
\tfrac{\partial^2 }{\partial x_j\partial x_k}
        f_\beta (t_{i-1},X_{t_{i-1}})
        \Delta {Z_{i-1}^{(j)}}\Delta {Z_{i-1}^{(k)}}
\big\|^2 
\Big]\\
& \leq 
C d^2 h^3.
\end{aligned}
\end{equation}
With regard to $\mathcal{I}_{6,3} $,
using the H\"older inequality, Assumption \ref{ass-f-3continuous} and Lemma \ref{lem:moments-bounded} gives 
\begin{equation}
\label{eq:I63-est}
\mathcal{I}_{6,3} 
\leq 
C d^3 h^3.
\end{equation}
Gathering estimates \eqref{eq:I61-est}, \eqref{eq:I62-est} and \eqref{eq:I63-est} together yields
\begin{equation}
\mathbb{E} 
\big[
\|\mathcal{I}_6\|^2
\big]
\leq 
C d^3 h^3.
\end{equation}
Based on all the above estimates, we conclude that 
\begin{equation}
\label{eq:R_i-est}
\mathbb{E}
\Big[
\big\|
\sum_{i=1}^{n}R_{i}
\big\|^2
\Big]
\leq 
C d^3 h^3
(\ln h)^2.
\end{equation}

\noindent
Invoking \eqref{eq:R_i-est}, \eqref{eq:error-propagation} and the discrete Gronwall inequality, 
we finally obtain
\begin{equation}
\mathbb{E}
\big[
\|E_n\|^2
\big]
\leq 
C d^3h^3(|\ln h|)^2 \exp 
\Big(
\sum_{i=1}^N
C h
\Big) 
\leq Cd^3h^3(|\ln h|)^2.
\end{equation}
The proof of Theorem \ref{thm-exact-drift} is completed.
\end{proof}

As a direct consequence, 
we obtain the following result on a required number of iterations or mixing time of the algorithm (\ref{eq-srk-appr}), whose proof is easy and thus omitted.
\begin{proposition}
    Suppose Assumptions \ref{ass-abso-cont}, \ref{ass-RN-lip} and \ref{ass-f-3continuous} hold. 
    To achieve a given precision level $\epsilon > 0$ 
    in the $L^2$-Wasserstein distance, 
    a required number of iterations of the algorithm (\ref{eq-srk-appr}) is of order 
    $\mathcal{O} \big(\tfrac{d}{\epsilon^{2/3}}(\ln\tfrac{d^{3/2}}{\epsilon})^{2/3}\big)$.
\end{proposition}

\noindent
\textbf{Examples with exact drift: Gaussian mixture distributions.}
\newline

When the target distribution $\mu$ has special structure,
such as Gaussian mixture distributions:
\begin{equation}
\mu = \sum_{i=1}^{\kappa} \theta_{i} \mathcal{N}(\alpha_{i},
\mathbf{\Sigma}_{i}),
\quad \sum_{i=1}^{\kappa} \theta_{i}=1
\quad \text{and} \quad 0 \leq \theta_{i} \leq 1, \quad i=1,\ldots,\kappa,
\end{equation}
we can calculate the exact drift analytically.
For the Gaussian mixture distributions, the drift term (\ref{eq-sfs-exact-dens}) can be rewritten as:
\begin{equation}
f_{\beta}(t, x) = \frac{\beta\sum_{i=1}^{\kappa}\theta_{i}\mathbb{E}_{\xi}\left[\nabla g_{\beta,i}(x+\sqrt{(1-t)\beta}\,\xi)\right]}
{\sum_{i=1}^{\kappa}\theta_{i}\mathbb{E}_{\xi}\left[g_{\beta,i}(x+\sqrt{(1-t)\beta}\,\xi)\right]},
\quad\xi\sim\gamma^{d}.
\end{equation}
Here $\kappa\in \mathbb{N}$ is the number of mixture components,
$\mathcal{N}(\alpha_{i},\Sigma_{i})$ is the $i$-th Gaussian component
with mean $\alpha_{i}\in \mathbb{R}^{d}$ and covariance matrix $\Sigma_{i}\in \mathbb{R}^{d\times d}$.
As shown in \cite{wang2025multimodal}, one has
\begin{equation*}
    \begin{aligned}
        &\mathbb{E}_{\xi}\left[\nabla g_{\beta,i}(x+\sqrt{(1-t)\beta}\,\xi)\right] \\
        =& \frac{\Sigma_{i}^{-1}\alpha_{i}+\left(\tfrac{1}{\beta}\mathbf{I}_{d}-\Sigma_{i}^{-1}\right)\left[t\mathbf{I}_{d}+(1-t)\beta\Sigma_{i}^{-1}\right]^{-1}\left[(1-t)\beta\Sigma_{i}^{-1}\alpha_{i}+x\right]}{\|t\Sigma_{i}+(1-t)\beta\mathbf{I}_{d}\|^{1/2}}b_{\beta,i}(t,x),
    \end{aligned}
\end{equation*}
\begin{equation*}
\mathbb{E}_{\xi}\left[g_{\beta,i}(x+\sqrt{(1-t)\beta}\,\xi)\right]=\frac{b_{\beta,i}(t,x)}{\|t\Sigma_{i}+(1-t)\beta\mathbf{I}_{d}\|^{1/2}},
\end{equation*}
where
\begin{equation}
\begin{aligned}
b_{\beta,i}(t, x)=&\exp\left(\tfrac{1}{2(1-t)\beta}\left\|\left(t\mathbf{I}_{d}+(1-t)\beta\Sigma_{i}^{-1}\right)^{-1/2}\left((1-t)\beta\Sigma_{i}^{-1}\alpha_{i}+x\right)\right\|^{2}\right)\\
&\cdot\exp\left(-
\tfrac{1}{2}\alpha_{i}^{T}\Sigma_{i}^{-1}\alpha_{i}-\tfrac{1}{2(1-t)\beta}\|x\|^{2}\right).
\end{aligned}
\end{equation}

\subsection{Accelerated Schr\"odinger-F\"ollmer sampler with inexact drift}
%
When the target distribution $\mu$ is complex, 
the exact drift $f_\beta$ is often computationally intractable. 
In this case,
we can use the Monte Carlo method to approximate the expectation, and thereby obtain an estimator 
of $ f_{\beta} $
defined by (\ref{eq-sfs-exact-dens}) with $ M $-sample mean.
Let $ \xi_{1},\cdots,\xi_{M} $ be i.i.d.\ $ \gamma^{d} $,
where $ M\geq 1 $ is sufficiently large. 
Therefore, we can approximate $ f_{\beta}(t,x) $ 
by the estimator $ \widetilde{f}_{\beta}^{M}:\Omega_{\xi}\times\mathbb{R}^{d}\times [0,1) \rightarrow\mathbb{R}^{d} $, 
\begin{equation}
\label{eq-sfs-mc-exact}
    \begin{aligned}
    \widetilde{f}_{\beta}^{M}(t,x) =& \frac{\tfrac{\beta}{M} 
    \sum\limits_{j=1}^{M} 
    \left[ \nabla g_{\beta} 
    \big(x + \sqrt{(1-t) \beta} 
    \, \xi_{j} \big) \right]}
    {\tfrac{1}{M} \sum\limits_{j=1}^{M}
    \left[ g_{\beta} \big( x + \sqrt{(1-t) \beta} \, \xi_{j} \big) \right]}\\
    =& 
    \frac{
    \tfrac{\sqrt{\beta}}{M} \sum\limits_{j=1}^{M} 
    \left[ \xi_j \cdot 
       g_{\beta} 
    \big(x + \sqrt{(1-t) \beta} 
    \, \xi_{j} \big)
    \right]}
    {\tfrac{\sqrt{(1-t)}
    }{M}
    \sum\limits_{j=1}^{M}
    \left[             
    g_{\beta} 
    \big(x + \sqrt{(1-t) \beta} 
    \, \xi_{j} \big)
    \right]},
    \quad \xi_{j} \sim \gamma^{d},        
    \end{aligned}
\end{equation}
where the second equality stands due to Stein's lemma.
In this setting, the SRK Schr\"odinger-F\"ollmer sampler with temperatures is given by
\begin{equation}
\label{eq-srk-mc-appr}
\begin{aligned}
\widetilde{Y}_{n+1}
&= 
\widetilde{Y}_{n} 
+ 
\tfrac{1}{3}
\widetilde{f}_{\beta}^{M}
(t_n, \widetilde{Y}_{n})h 
+ 
\tfrac{2}{3}
\widetilde{f}_{\beta}^{M}
\big(t_n+\tfrac{3}{4}h, \widetilde{H}_{n}\big)h+
\sqrt{\beta}
\Delta {W_n},
\end{aligned}
\end{equation}
with
\[
\widetilde{H}_{n} = \widetilde{Y}_{n} + \tfrac{3}{4}\widetilde{f}_{\beta}^{M}(t_n, \widetilde{Y}_{n})h + \tfrac{3\sqrt{\beta}\Delta {Z_n}}{2h}.
\]

For the purpose of the error analysis, we show that the inexact drift $\widetilde{f}^M_\beta$ is also Lipschitz continuous.
\begin{proposition}
\label{prop-sfs-mc-f-lip}
    Let Assumptions \ref{ass-abso-cont} and \ref{ass-RN-lip} hold. Then 
    for $t\in [0,1)$, the estimator $\widetilde{f}^M_{\beta} \colon
    (t, \cdot)
    \rightarrow \mathbb{R}^d$ is uniformly Lipschitz continuous:
    \begin{align}
    \label{eq-sfs-mc-f-lip}
        \big\|
        \widetilde{f}^M
        _{\beta}(t,x)
        -
        \widetilde{f}^M
        _{\beta}(t,y)
        \big\|
        \leq
        L_f\|x-y\|,
        \quad
        \forall \, x,y \in \mathbb{R}^d,
    \end{align} 
    { where $L_f := \big(1+\tfrac{L_g}{\rho}
    \big)
    \tfrac{\beta L_g}{\rho}$ is independent of $d$ and $h$.}
\end{proposition}

The statistical error caused by the Monte Carlo approximation is quantified by the following lemma (see \cite[Lemma A.6]{huang2025schrodinger} for a similar proof).
\begin{lemma}
\label{lem-mc-error}
Under Assumptions  \ref{ass-abso-cont} and \ref{ass-RN-lip},
for any $x \in \mathbb{R}^d$ and $t\in[0,1)$, 
\begin{equation}
\label{eq-sfs-f-mc-error}
\mathbb{E} 
\big[ 
\| \widetilde{f}_{\beta}^{M}(t, x) - 
f_\beta(t, x) \|^2 
\big]
\leq 
\tfrac{C_2 d}{M},
\end{equation}
where $C_2 := \tfrac{4L_g^2\beta^2}{\rho}\big( 1+\tfrac{L_g^2}{\rho^2} \big)$, 
$L_g, \rho$ come from Assumption \ref{ass-RN-lip}, 
and $M$ is the number of samples
used in the Monte Carlo estimator \eqref{eq-sfs-mc-exact}.
\end{lemma}
Now, we are able to establish the error bound for the SRK Schr\"odinger-F\"ollmer sampler \eqref{eq-srk-mc-appr}. 
\begin{thm}
\label{thm-mc-error-bound}
(Main result: error bounds with inexact drift)
Let Assumptions \ref{ass-abso-cont}, \ref{ass-RN-lip} and \ref{ass-f-3continuous} hold. 
Let $\{\widetilde{Y}_n\}_{n \in[N]_0}$ be 
defined by \eqref{eq-srk-mc-appr} with the uniform step size $h=1/N$. 
Then there exists a constant $C$ independent of $d,h$ and $M$, such that,
\begin{align*}
\mathcal{W}_2
\big(
\text{Law}(\widetilde{Y}_1), \mu
\big) 
\leq 
C (d h)^{3/2} |\ln h| 
+
C \sqrt{\tfrac{d}{M}},
\end{align*}
where $M$ is the number of samples used in the Monte Carlo estimator
\eqref{eq-sfs-mc-exact}.
\end{thm}

\begin{proof}
Using the triangle inequality, we first decompose the total error as follows:
\begin{equation}
    \mathcal{W}_2 
    \big(
    \text{Law}(\widetilde{Y}_1), \mu
    \big) 
    \leq 
    \mathcal{W}_2 
    \big(
    \text{Law}(\widetilde{Y}_1), \text{Law}(Y_1)
    \big) + \mathcal{W}_2 
    \big(
    \text{Law}(Y_1), \mu
    \big).
\end{equation}
From Theorem \ref{thm-exact-drift} it follows that
\[
\mathcal{W}_2 
    \big(
    \text{Law}(Y_1), \mu
    \big)
    \leq
    C (d h)^{3/2} |\ln h|.
\]
Consequently, it remains to estimate the first term, 
which captures the propagation of Monte Carlo error through the SRKSFS \eqref{eq-srk-mc-appr}. 
We subtract \eqref{eq-srk-mc-appr} from \eqref{eq-srk-appr} to obtain
\begin{equation}
\begin{aligned}
\widetilde{Y}_n
-
Y_n
&= 
\widetilde{Y}_{n-1}
-
Y_{n-1}
+
\tfrac{h}{3}
\Big( 
\widetilde{f}_\beta^M (t_{n-1},\widetilde{Y}_{n-1}) 
-
f_\beta(t_{n-1},Y_{n-1}) \Big)
\\
&\qquad+
\tfrac{2h}{3}\Big(\widetilde{f}_\beta^M (t_{n-1}+\tfrac{3h}{4},\widetilde{H}_{n-1}) - f_\beta(t_{n-1}+\tfrac{3h}{4},H_{n-1}) \Big)
\\
&=
\tfrac{h}{3}
\sum_{i=1}^{n}
\Big( 
\widetilde{f}_\beta^M (t_{i-1},\widetilde{Y}_{i-1}) 
-
f_\beta(t_{n-1},Y_{n-1}) \Big)
\\
&\qquad
+\tfrac{2h}{3}\sum_{i=1}^n
\Big(
\widetilde{f}_\beta^M (t_{i-1}+\tfrac{3h}{4},\widetilde{H}_{i-1}) 
- 
f_\beta(t_{i-1}+\tfrac{3h}{4},H_{i-1}) 
\Big),
\end{aligned}
\end{equation}
where we used the fact that $\widetilde{Y}_0=Y_0=0$.
Squaring both sides of the above equation, taking expectations and using the Lipschitz condition \eqref{eq-sfs-mc-f-lip}, we arrive at
\begin{equation}
\begin{aligned}
&
\mathbb{E}
\big[
\|\widetilde{Y}_n - Y_n\|^2
\big]\\
& \leq
\tfrac{2h}{9}
\sum_{i=1}^n
\E
\big[
\|\widetilde{f}_\beta^M (t_{i-1},\widetilde{Y}_{i-1}) 
-
f_\beta(t_{i-1},Y_{i-1})\| ^2
\big]
\\
& \quad +
\tfrac{8h}{9}
\sum_{i=1}^n
\E
\big[
\|\widetilde{f}_\beta^M (t_{i-1}+\tfrac{3h}{4},\widetilde{H}_{i-1}) - f_\beta(t_{i-1}+\tfrac{3h}{4},H_{i-1})\| ^2
\big]
\\
&\leq 
\tfrac{4L_f^2h}{9}
\sum_{i=1}^n
\mathbb{E}
\big[
\|\widetilde{Y}_{i-1} - Y_{i-1}\|^2
\big]
+
\tfrac{4h}{9}
\sum_{i=1}^n
\E
\big[
\|\widetilde{f}_\beta^M (t_{i-1},{Y}_{i-1}) 
-
\widetilde{f}_\beta(t_{i-1},{Y}_{i-1})\| ^2
\big]\\
& \quad +
\tfrac{16L_f^2h}{9}
\sum_{i=1}^n
\mathbb{E}
\big[
\|\widetilde{H}_{i-1} - H_{i-1}\|^2
\big]
+
\tfrac{16h}{9}
\sum_{i=1}^n
\E
\big[
\|\widetilde{f}_\beta^M 
(t_{i-1}+\tfrac{3h}{4},{H}_{i-1}) 
- 
f_\beta(t_{i-1}+\tfrac{3h}{4},{H}_{i-1})\|^2
\big]\\
& \leq
\big(
2L_f^2 +\tfrac{4}{9}
\big)
L_f^2 h
\sum_{i=1}^n
\mathbb{E}
\big[
\|\widetilde{Y}_{i-1} - Y_{i-1}\|^2
\big]
+
\big(
2L_f^2 +\tfrac{4}{9}
\big)
 h
\sum_{i=1}^n
\E
\big[
\|\widetilde{f}_\beta^M (t_{i-1},{Y}_{i-1}) 
-
f_\beta(t_{i-1},{Y}_{i-1})\| ^2
\big]\\
& \quad +
\tfrac{16h}{9}
\sum_{i=1}^n
\E
\big[
\|\widetilde{f}_\beta^M (t_{i-1}+\tfrac{3h}{4},{H}_{i-1}) 
- 
f_\beta(t_{i-1}+\tfrac{3h}{4},{H}_{i-1})\|^2
\big].
\end{aligned}
\end{equation}
For the estimate of the last second term of the above inequality, 
we use the property of the conditional expectation \cite[Theorem 2.24]{klebaner2012introduction} and Lemma \ref{lem-mc-error} to get, 
for any $i \in [N-1]$,
\begin{equation}
\begin{aligned}
\mathbb{E} 
\big[ 
\| \widetilde{f}_{\beta}^{M}(t_i, {Y}_i)
-
{f}_\beta(t_i, {Y}_i)\|^2 
\big]
& = 
\E
\Big[
\E
\big[ 
\| \widetilde{f}_{\beta}^{M}(t_i, {Y}_i)
-
{f}_\beta(t_i, {Y}_i)\|^2
\big|
{Y}_i
\big]
\Big]\\
& =
\E_W\Big[\E_\xi\big[ 
\| \widetilde{f}_{\beta}^{M}(t_i, {Y}_i)
-
{f}_\beta(t_i, {Y}_i)\|^2
\big]\Big]\\
& \leq 
\tfrac{C_2 d}{M},
\end{aligned}
\end{equation}
and 
\begin{equation}
\E
\big[
\|\widetilde{f}_\beta^M 
(t_{i-1}+\tfrac{3h}{4},{H}_{i-1}) 
- 
f_\beta(t_{i-1}+\tfrac{3h}{4},{H}_{i-1})
\|^2
\big] 
\leq 
\tfrac{C_2 d}{M}.
\end{equation}

\noindent
Equipped with the above estimates, one can derive from the discrete Gronwall inequality 
that 
\begin{equation}
\mathbb{E}
\big[
\|\widetilde{Y}_1 - Y_1\|^2
\big]
\leq
\tfrac{Cd}{M}.
\end{equation}
The proof of Theorem \ref{thm-mc-error-bound} is completed.
\end{proof}
As an immediate corollary, we get the computational complexity of the sampler to achieve a prescribed accuracy.
\begin{proposition}
\label{prop:mixing-time}
    Suppose Assumptions \ref{ass-abso-cont}, \ref{ass-RN-lip} and \ref{ass-f-3continuous} hold. 
    To achieve a given precision level $\epsilon > 0$ 
    in the $L^2$-Wasserstein distance, 
    a required number of iterations of the algorithm (\ref{eq-srk-mc-appr}) is of order 
    $\mathcal{O} \big(\tfrac{d}{\epsilon^{2/3}}(\ln\tfrac{d^{3/2}}{\epsilon})^{2/3}\big)$,
    while the number of samples required for the Monte Carlo estimator 
    is of order
    $\mathcal{O} 
    \big( 
    \tfrac{d}{\epsilon^2} 
    \big)$.
\end{proposition}
Its proof is deferred to Section \ref{appendix_mixing-time}.
In Table \ref{tab:comparison}, we compare our results with \cite{huang2025schrodinger,wang2025multimodal}, in terms of error
bounds, the number of iterations of SRKSFS and the number of samples used in the Monte Carlo estimator of the drift required to achieve the accuracy
tolerance $\epsilon$ in $\mathcal{W}_2$ distance. 
\begin{table}[h]
\caption{A comparison of non-asymptotic error bounds in $\mathcal{W}_2$-distance.}
\vspace{-2pt} 
\centering
\renewcommand{\arraystretch}{1.5} 
\footnotesize 
\begin{tabular}{ccccccc}
\Xhline{1pt} 
 & \makecell[c]
 {Strong \\ convexity} 
 & \makecell[c]{Additional\\ condition\textsuperscript{1}}
 & \makecell[c]{Error\\ bound} 
 & \makecell[c]{{ Number of} \\{  iterations}}
 & \makecell[c]{{ Number of}
 \\{ MC samples}}
 \\ \hline
\cite{huang2025schrodinger} 
& Yes
& No 
& $\mathcal{O}(\sqrt{dh})+\mathcal{O}(\sqrt{d/M})$
& { 
$\mathcal{O}(d/\epsilon^{2})$}
& { $\mathcal{O}(d^2/\epsilon^{4})$ }
\\ 
\cite{wang2025multimodal}
& No 
& Yes 
& $\mathcal{O}(dh)+\mathcal{O}(\sqrt{d/M})$ 
& { $\mathcal{O}(d/\epsilon)$}
& { $\mathcal{O}(d/\epsilon^{2})$}\\
This work 
& No 
& Yes
& 
$\mathcal{O}((dh)^{3/2}| \ln h|)+\mathcal{O}(\sqrt{d/M})$
& $\mathcal{O} (d/\epsilon^{2/3}(\ln\tfrac{d^{3/2}}{\epsilon})^{2/{3}})$
& $\mathcal{O}(d/\epsilon^2)$
\\
\Xhline{1pt} 
\end{tabular}
\label{tab:comparison}
\caption*{\textsuperscript{1}
Smoothness assumptions other
than the Lipschitz condition for the drift.}
\end{table}

\section{Numerical experiments for SRKSFS: known densities}
\label{sec:experiments}
\noindent
To evaluate the efficiency and quality of the proposed sampling algorithms, 
we conduct a series of numerical experiments in this section. 
Given our focus on multimodal distributions, 
we first examine several bivariate Gaussian mixture models where the modes can be easily visualized. 
Subsequently, we study more complex multivariate distributions of arbitrary dimensions using copula modeling \cite{nelsen2006introduction}. 
Finally, we test our methods on deep generative models for image synthesis.

\subsection{Gaussian Mixture Models}
\label{subsec:Gaussian-Mixture-Models}
In this subsection, we consider sampling from three types of bivariate Gaussian mixture distributions. 
The general form of a Gaussian mixture density is given by:
\begin{equation}
    p(x) = \sum_{i=1}^{\kappa} \theta_{i} \mathcal{N}(x; \alpha_{i}, \Sigma_{i}), \quad \sum_{i=1}^{\kappa} \theta_{i}=1 \quad \text{and} \quad 0 \leq \theta_{i} \leq 1, \quad i=1,\ldots,\kappa,
\end{equation}
where $\alpha_i\in \mathbb{R}^d$ and $\Sigma_i\in \mathbb{R}^{d \times d}$ are the mean and covariance matrix of Gaussian component, respectively.
\paragraph{Gaussian Circle}
First, we test the Gaussian Circle distribution with $\kappa = 8$ components. 
The means are located at $\alpha_i = 4(\cos(2\pi(i-1)/8), \sin(2\pi(i-1)/8))$, 
and the covariance matrices are ${\Sigma}_{i} = 0.3 \mathbf{I}_{2}$ for $i = 1, \dots, 8$. 
Figure \ref{fig-circle-samp} displays scatter plots obtained with different step sizes.
Numerical results demonstrate that the stochastic Runge-Kutta scheme achieves accurate sampling even with larger step sizes, 
whereas the Euler-Maruyama scheme requires significantly smaller step sizes to achieve comparable accuracy. 
\begin{figure}[htbp]
    \centering
    \includegraphics[width=\linewidth]{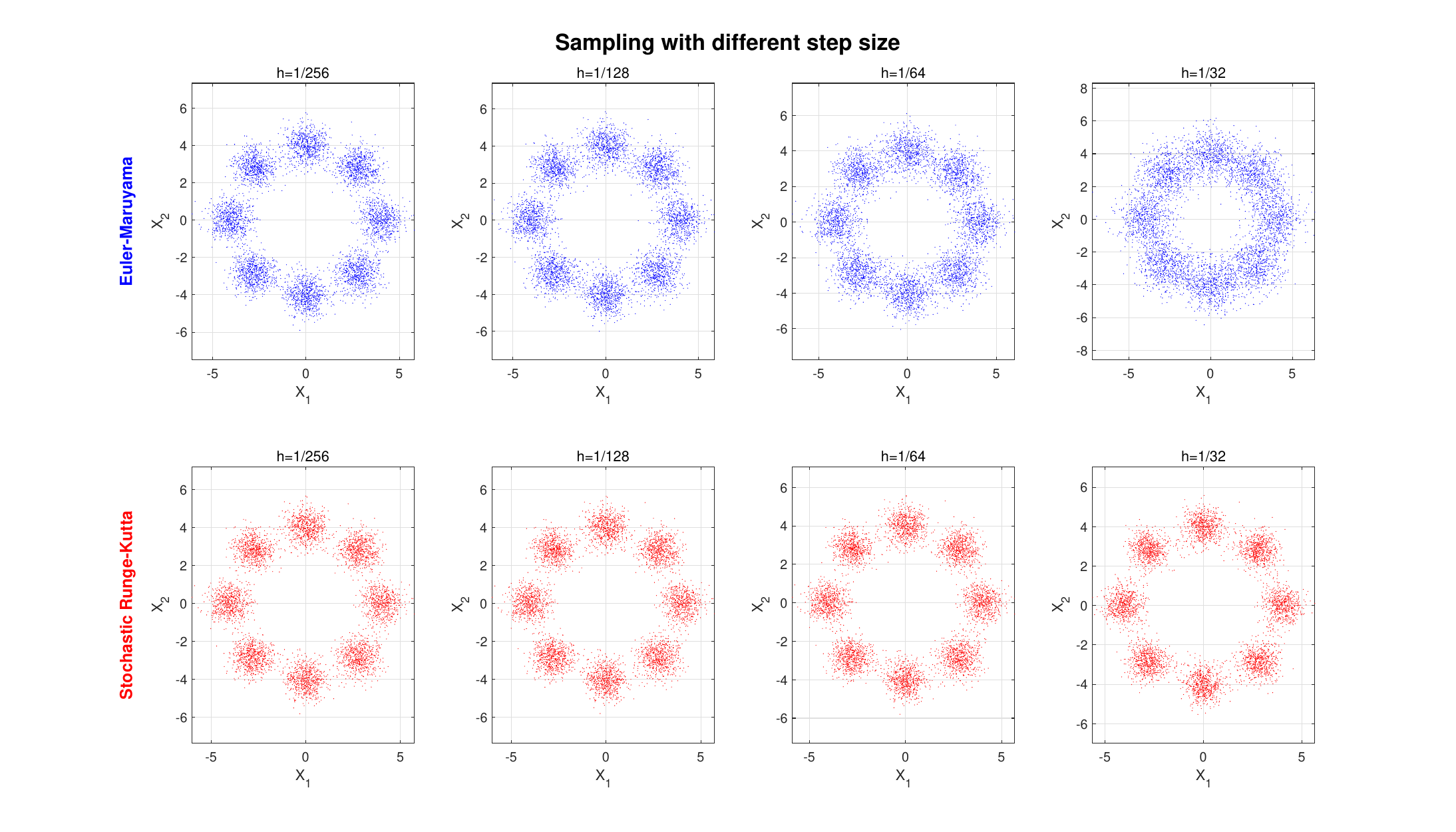}
    \caption{Sampling of Gaussian Circle with different step sizes.}
    \label{fig-circle-samp}
\end{figure}
 
Also, we compute the root mean square error for various step sizes from $h = 2^{-10}$ to $h = 2^{-5}$, 
and the convergence rates are shown in Figure \ref{fig-circle-cov}. Here we set the solution with step size $h = 2^{-13}$ as the exact reference. 
It is shown that, the convergence of the Euler-Maruyama method with order $1$ deteriorates when the step size is large ($h = 1/64$), but the stochastic Runge-Kutta scheme still maintains order $1.5$.
\begin{figure}[htbp]
    \centering
    \includegraphics[width=0.8\linewidth]{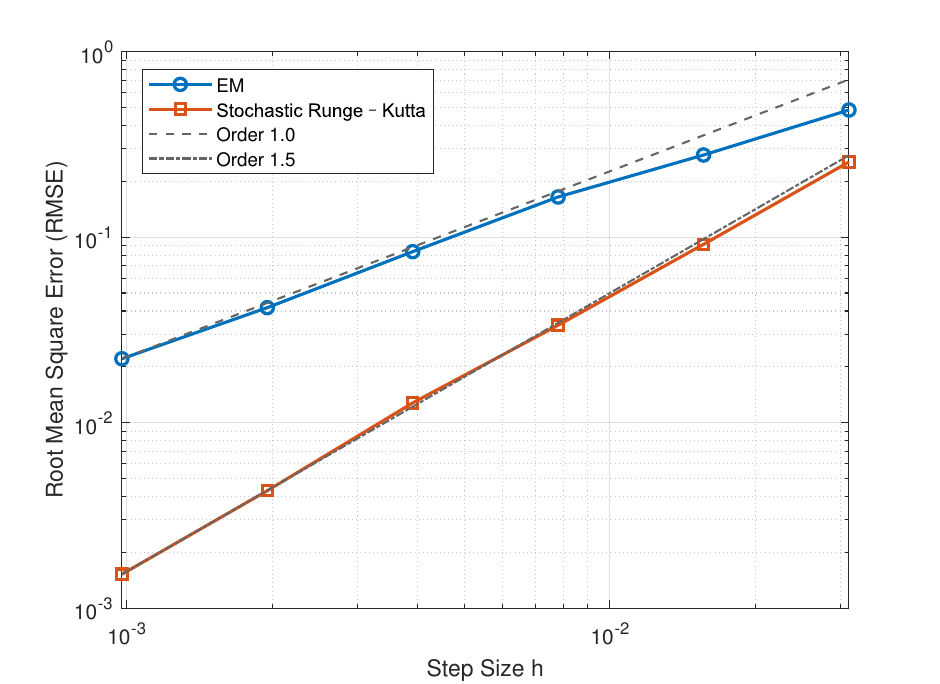}
    \caption{Mean-square convergence rates of Gaussian Circle with different schemes.}
    \label{fig-circle-cov}
\end{figure}

\paragraph{Gaussian Cross}
As the second example, we examine the Gaussian Cross distribution, 
a mixture model composed of eight Gaussian components arranged in a cruciform pattern. 
The mixture is constructed by positioning components at four central locations $\alpha_i \in \{(\pm 1.5, 0), (0, \pm 1.5)\}$. 
At each location, we define two components sharing the same mean but possessing distinct covariance structures, given by
${\Sigma}_{i} = 
\begin{pmatrix}
      1 & \pm0.9 \\
      \pm0.9 & 1
\end{pmatrix}
$.
Here, the off-diagonal elements are chosen to be either $+0.9$ or $-0.9$, 
generating strongly correlated components with opposing orientations at each positional mean. 
All components are assigned equal mixing weights $\theta_i = 1/8$. 
This configuration creates a challenging multimodal target, testing the ability of the sampling algorithm to traverse between modes separated along the coordinate axes and to accurately capture the distinct correlation structures within each mode.

The mean-square convergence rates of the Euler-Maruyama scheme 
and the stochastic Runge-Kutta scheme are shown in Figure \ref{fig-cross-cov}, 
which align with the theoretical expectations.
\begin{figure}[htbp]
    \centering
    \includegraphics[width=0.8\linewidth]{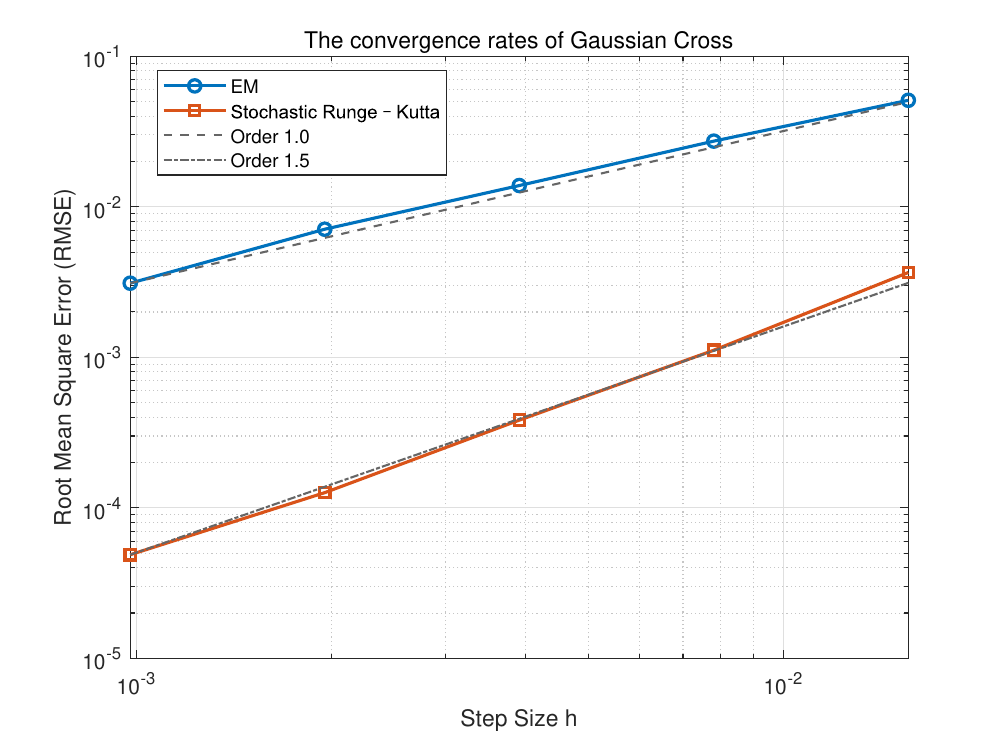}
    \caption{Mean-square convergence rates of Gaussian Cross under different schemes.}
    \label{fig-cross-cov}
\end{figure}
Figure \ref{fig-cross-samp} displays scatter plots and contours for different dimensions obtained using different sampling methods, 
showing that the SFS method accurately captures the structural characteristics.
\begin{figure}[htbp]
    \centering
    \includegraphics[width=\linewidth]{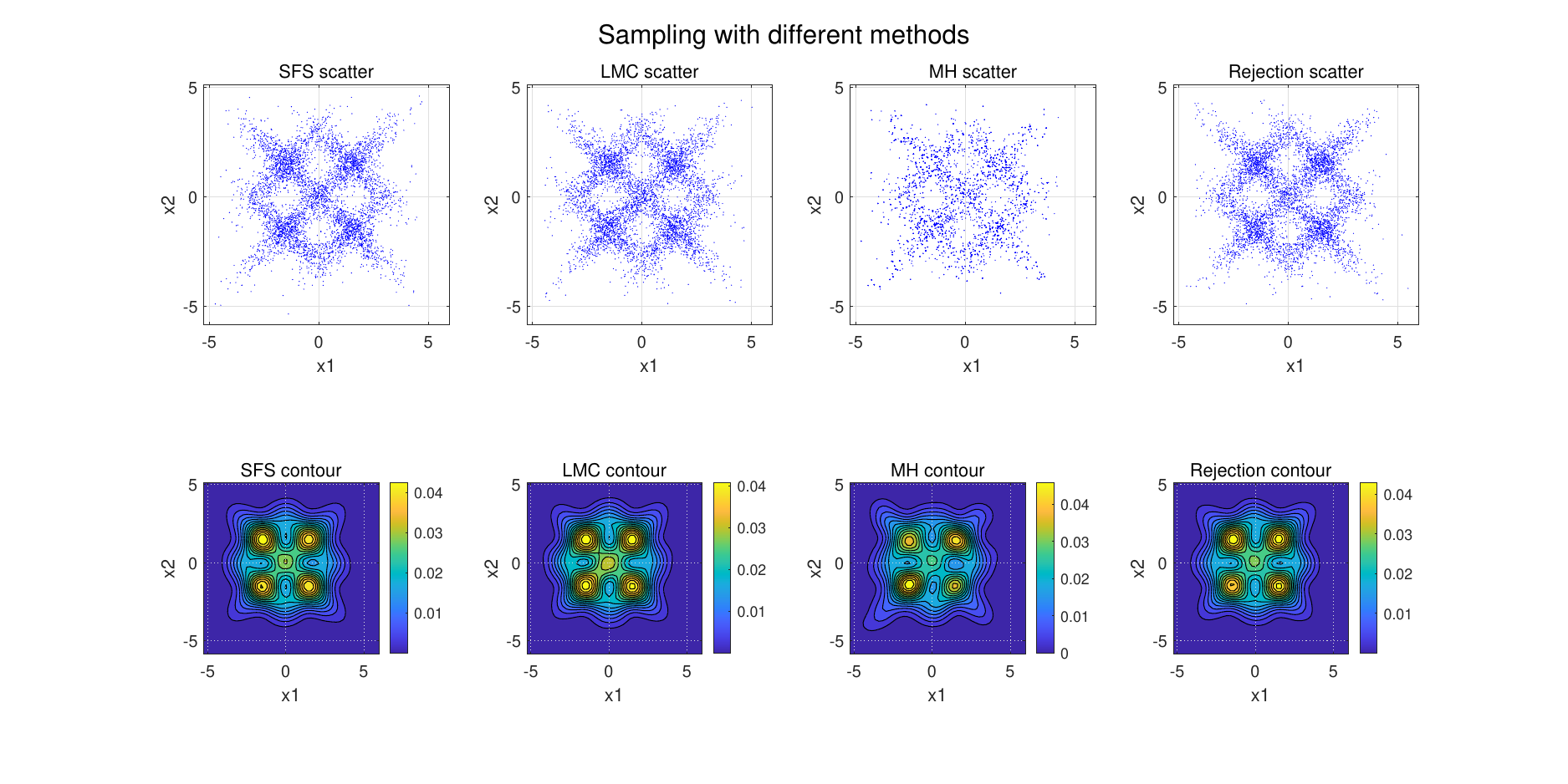}
    \caption{Sampling of Gaussian Cross using different algorithms.}
    \label{fig-cross-samp}
\end{figure}

\paragraph{Circular Gaussian Mixture}
Finally, we examine the Circular Gaussian mixture distribution, defined by the density:

\begin{equation}
p(x) = \tfrac{1}{3} \sum_{i=1}^{3} \tfrac{1}{2 \pi R_{i} \sigma \sqrt{2 \pi}} \exp \Big( -\tfrac{ ( r - R_{i} )^{2} }{ 2 \sigma^{2} } \Big),
\end{equation}
where $r = 
\sqrt{ x^{2} + y^{2} }$, $R = [1,2,3]$, 
and $\sigma = 0.1$. 
This model describes a system comprising three concentric rings, 
serving as a standard test case for multimodal sampling algorithms.
Figure \ref{fig-ring-samp} displays scatter plots and marginal distributions along different dimensions obtained using various sampling methods. 
Numerical results also show that the SFS method approximates the exact distribution well and outperforms the other algorithms.
\begin{figure}[htbp]
    \centering
    \includegraphics[width=\linewidth]{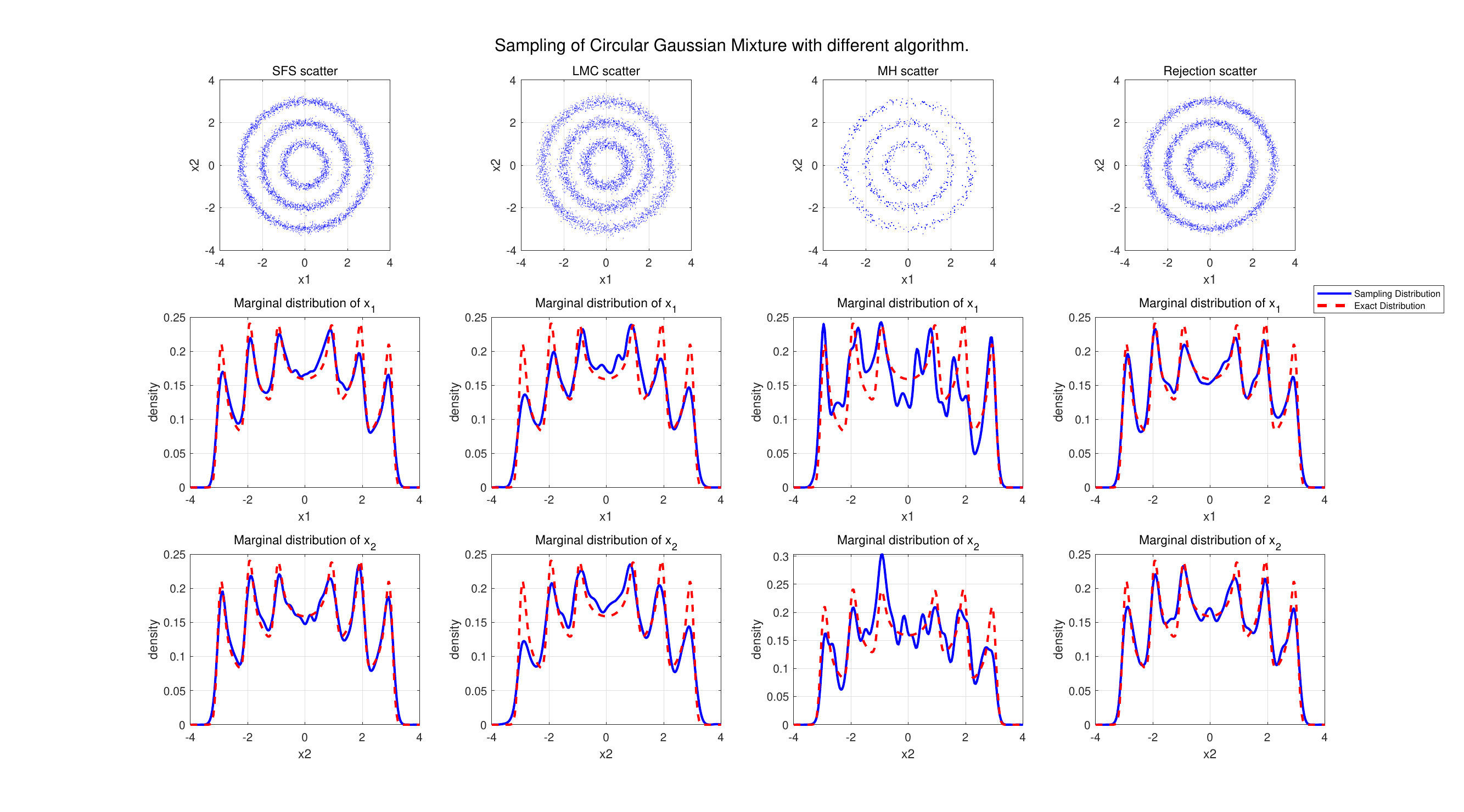}
    \caption{Sampling of the Circular Gaussian Mixture using different algorithms.}
    \label{fig-ring-samp}
\end{figure}

\subsection{Copula-Generated Distributions}

To study more general multimodal distributions in higher dimensions, 
we come to copula modeling \cite{nelsen2006introduction}. 
Copulas provide a flexible framework for constructing multivariate distributions with arbitrary marginal distributions and dependency structures, 
making them ideal for evaluating sampling algorithms on complex, high-dimensional targets.

For ${x} = (x_1, \dots, x_d)^\top \in \mathbb{R}^d$, 
the joint density $p({x})$ is defined via a Clayton copula with parameter $\theta = 2 $ 
and identical bimodal Gaussian mixture marginals on each dimension. 
This parameter choice induces moderate to strong positive dependence between variables, 
particularly in the lower tail, 
creating a challenging sampling scenario with complex correlation structures.

For each dimension $i = 1, \cdots, d $, 
the marginal cumulative distribution function (CDF) is:
\begin{equation}
u_i = F_{X_i}(x_i) = \sum_{k=1}^{2} w_k  \Phi \left( \tfrac{x_i - \mu_k}{\sigma_k} \right), 
\end{equation}
and the marginal probability density function (PDF) is:
\begin{equation}
f_{X_i}(x_i) = \sum_{k=1}^{2} w_k  \tfrac{1}{\sigma_k} \phi \left( \tfrac{x_i - \mu_k}{\sigma_k} \right),
\end{equation}
where $\Phi(\cdot)$ and $\phi(\cdot)$ denote the standard normal CDF and PDF, respectively, 
and the mixture parameters are identical across all dimensions:
\begin{equation}
w_1 = 0.7,\quad w_2 = 0.3,\quad \mu_1 = -1,\quad \mu_2 = 1,\quad \sigma_1 = \sigma_2 = 0.2.
\end{equation}
The $d$-dimensional Clayton copula density is:
\begin{equation}
c_\theta(u_1, \dots, u_d) = (1 + \theta)^{d - 1}
\left( \prod_{i=1}^d u_i^{-(1 + \theta)} \right)
\left( \sum_{i=1}^d u_i^{-\theta} - d + 1 \right)^{-\left( \frac{1}{\theta} + d \right)}.
\end{equation}
Consequently, the full joint density is:
\begin{equation}
p(\mathbf{x}) = c_\theta\big(F_{X_1}(x_1), \dots, F_{X_d}(x_d)\big) \cdot \prod_{i=1}^d f_{X_i}(x_i).
\end{equation}

This construction yields a distribution with complex multimodal structure. 
While each marginal distribution has two modes, 
the Clayton copula induces asymmetric dependence that creates intricate interactions between dimensions. 
The resulting joint distribution exhibits multiple concentration regions whose exact count depends on both the marginal bimodality and the copula parameter $\theta$. 
These concentration regions are not simply the product of marginal modes due to the nonlinear dependence structure, 
making this distribution particularly challenging for sampling algorithms that struggle with complex correlations.
For the low dimension $d = 2$, Figure \ref{fig-clayton-samp-2} shows sampling results for four algorithms and
all successfully sample from the density.
\begin{figure}[htbp]
    \centering
    \includegraphics[width=\linewidth]{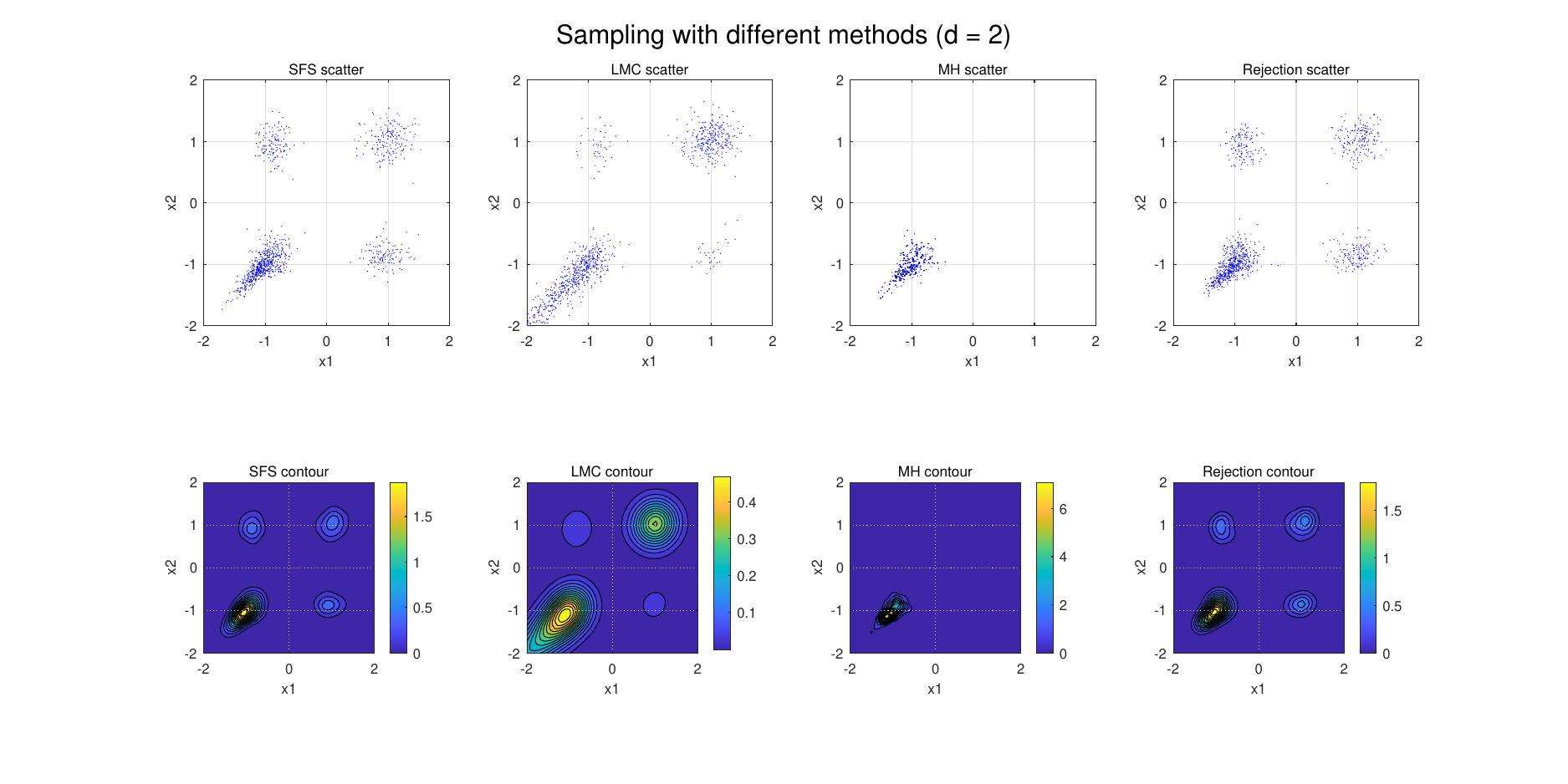}
    \caption{Sampling of the Clayton copula ($d = 2$) using different algorithms.}
    \label{fig-clayton-samp-2}
\end{figure}
However, when the dimension increases to $d = 5$, 
sampling becomes difficult.
The acceptance rate for rejection sampling drops sharply to $1.5\times 10^{-4}$, 
requiring significant computational time.
\begin{figure}[htbp]
    \centering
    \includegraphics[width=\linewidth]{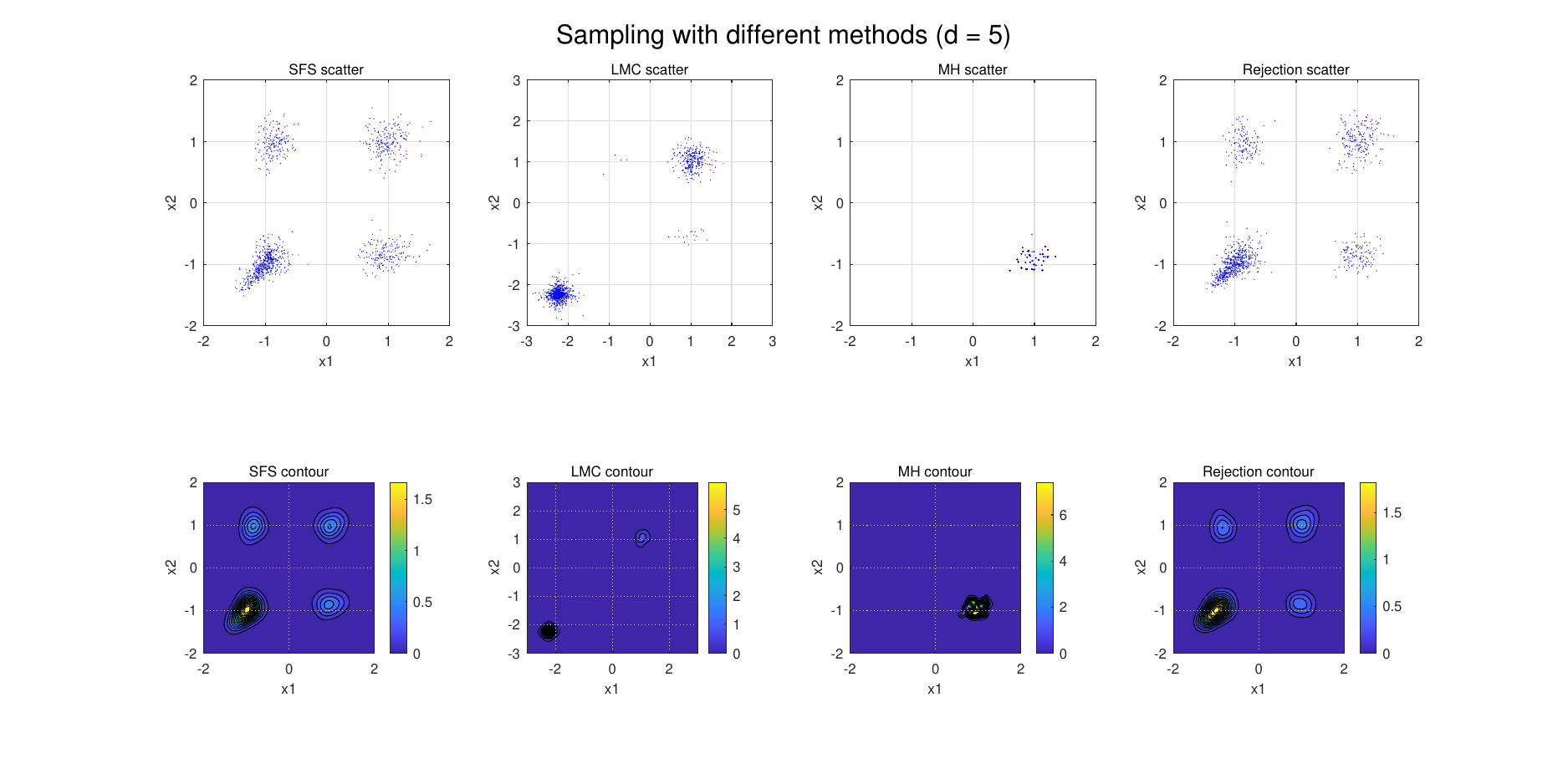}
    \caption{Sampling of the Clayton copula ($d = 5$) using different algorithms.}
    \label{fig-clayton-samp-5}
\end{figure}
Figure \ref{fig-clayton-samp-5} displays pairwise scatter plots and marginal densities for selected dimensions.
The SRKSFS \eqref{eq-srk-mc-appr} successfully captures the complex multimodal structure and preserves the asymmetric dependence induced by the Clayton copula. 
In contrast, the Langevin dynamics sampler fails to adequately explore regions of high density, 
particularly in the lower tail where dependence is the strongest.

\subsection{Deep Generative Models}
Modern deep learning techniques have attracted enormous attention from statistical researchers 
and practitioners, among which deep generative models are a class of important unsupervised learning methods 
\cite{salakhutdinov2015learning,bond2021deep}. 
Deep generative models attempt to model the statistical distribution of high-dimensional data using DNNs, 
with wide applications in image synthesis and text generation.

One general class of deep generative models has the form $X = G(Z)$, 
where $X\in \mathbb{R}^p$ is the high-dimensional data point, 
for example, an image, 
$Z\in \mathbb{R}^d$ is a latent random vector with $d \ll p$, 
and $G:\mathbb{R}^d \to \mathbb{R}^p$ is a DNN generator.
The distribution of $Z$ is characterized by an energy function $E:\mathbb{R}^d \to \mathbb{R}$, 
which is also a DNN. 
The pair $(E_d, G_{d,p})$ thus defines a deep generative model, 
where the subscripts $d$ and $p$ indicate the dimensions\cite{che2020your,pang2020learning}.

In this section, the functions $E$ and $G$ are supposed to be known and we focus on the sampling of $p(z) \varpropto  \exp (-E(z) )$, 
as it is the key to generating new data points of $X$.
We consider generative models for the Fashion-MNIST dataset 
\cite{xiao2017fashion}, 
which contains 60,000 training images and 10,000 test images, 
each consisting of $28 \times 28$ grey-scale pixels. 
We implement two different latent space configurations: 
a severely constrained model with $d = 2$ and a more expressive model with $d = 32$.

Figure \ref{fig-fashion-mnist-2} shows samples generated after sampling from the latent space $d=2$, 
where all algorithm successfully captures the multimodal structure of the 2D latent space, 
as evidenced by the distinct clusters in the latent space visualization.
\begin{figure}[htbp]
    \centering
    \begin{subfigure}[b]{0.32\textwidth}
        \includegraphics[width=\textwidth]{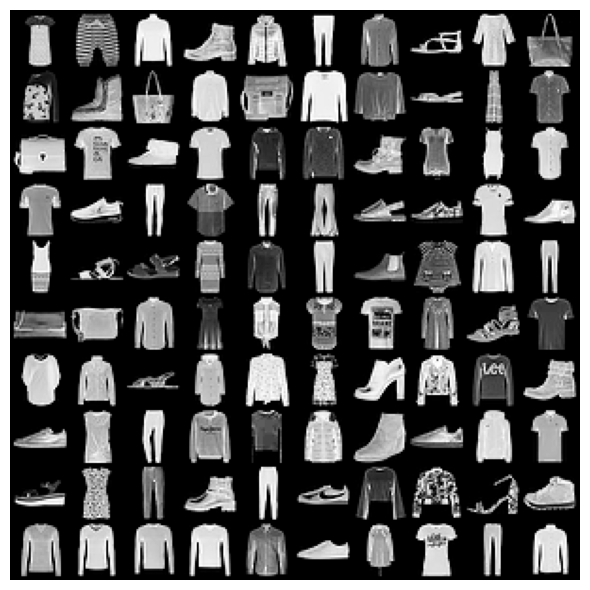}
        \caption{Original Samples}
    \end{subfigure}
    \hfill
    \begin{subfigure}[b]{0.32\textwidth}
        \includegraphics[width=\textwidth]{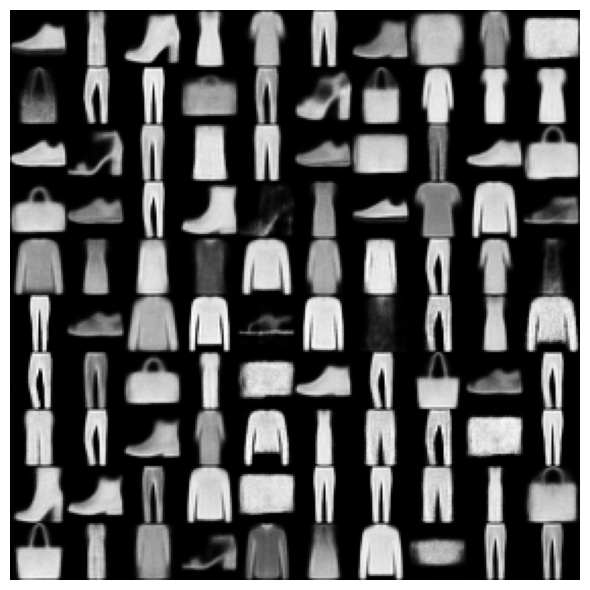}
        \caption{SFS Samples}
    \end{subfigure}
        \hfill
    \begin{subfigure}[b]{0.32\textwidth}
        \includegraphics[width=\textwidth]{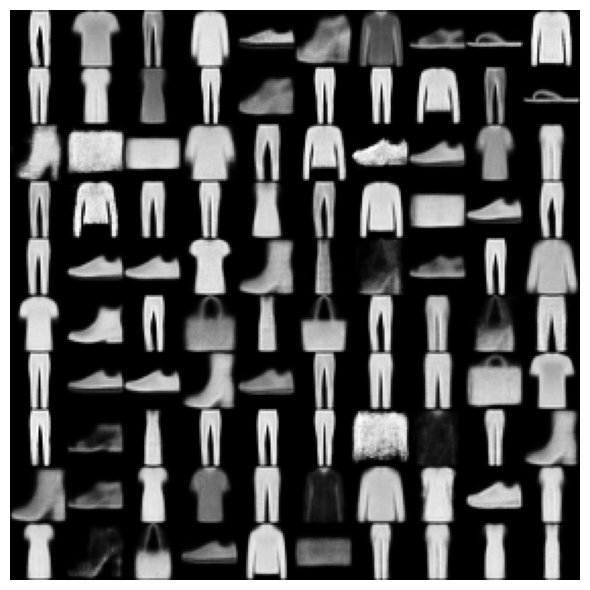}
        \caption{Langevin Samples}
    \end{subfigure}
    \caption{Images generated by a deep generative model with latent dimension $d = 2$.}
    \label{fig-fashion-mnist-2}
\end{figure}
While for the latent space $d=32$, as shown in Figure \ref{fig-fashion-mnist-32}, 
the advantage of the stochastic Runge-Kutta scheme becomes more pronounced.
\begin{figure}[htbp]
    \centering
    \begin{subfigure}[b]{0.32\textwidth}
        \includegraphics[width=\textwidth]{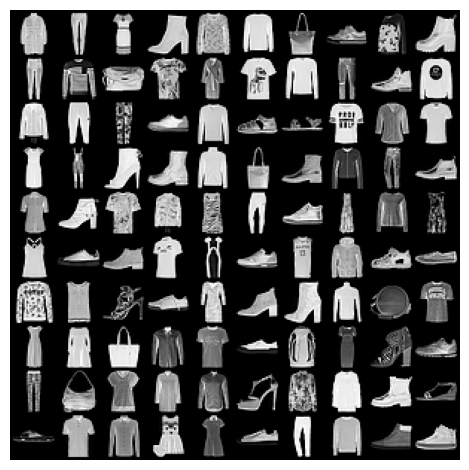}
        \caption{Original Samples}
    \end{subfigure}
    \hfill
    \begin{subfigure}[b]{0.32\textwidth}
        \includegraphics[width=\textwidth]{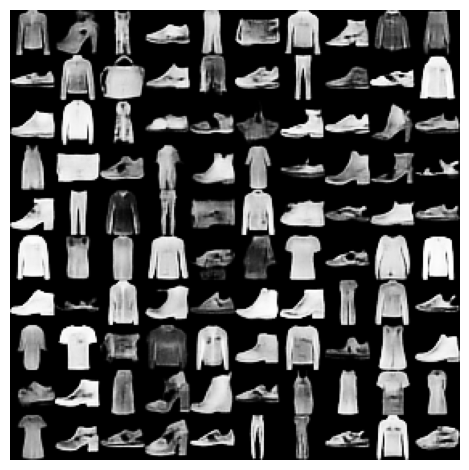}
        \caption{SFS Samples}
    \end{subfigure}
        \hfill
    \begin{subfigure}[b]{0.32\textwidth}
        \includegraphics[width=\textwidth]{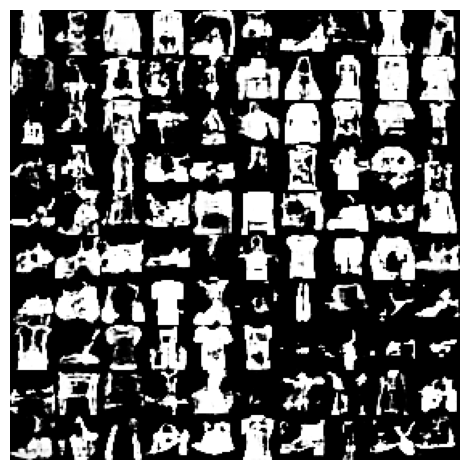}
        \caption{Langevin Samples}
    \end{subfigure}
    \caption{Images generated by a deep generative model with latent dimension $d = 32$.}
    \label{fig-fashion-mnist-32}
\end{figure}


\section{Data-driven Schr{\"o}dinger-F{\"o}llmer generation with empirical measures}
\label{sec:Data-driven-generation}
\subsection{Schr{\"o}dinger-F{\"o}llmer process with empirical measures}
In many real-world sample generation tasks, 
the target distribution $\mu(x)$ for $x\in \R^d$ is typically unknown,  
but finite number of independent samples  $\{\eta^{(i)}\}^{n}_{i=1} \sim \mu$ are provided. 
In this setting, can one extend the previously proposed sampler to generate new samples?
To answer this question, we recast the Radon-Nikodym derivative $g_\beta$ defined by \eqref{eq-sfs-g-beta} as follows: 
\[
    g_{\beta}(x)
    :=
    \tfrac{\mathrm{d}\mu}
    {\mathrm{d} \mathcal{N} (0,\beta\,\mathbf{I}_d)}(x)
    =
    \tfrac{(2\pi\beta)^{d/2}}{C}
    p(x)
    \exp 
    \big(
    \tfrac{\|x\|^2}{2\beta}
    \big), 
    \quad x \in \mathbb{R}^d.
\]
Using the change of variable $\eta = x+\sqrt{(1-t)\beta}\,\xi$, we obtain 
\begin{equation}
\begin{aligned}
& 
\mathbb{E}_{\xi \sim \gamma^{d}} 
\big[ \xi\,  g_{\beta}(x+\sqrt{(1-t)\beta}\,\xi ) \big] \\
& = 
\tfrac{\beta^{d/2}}{C}
\int_{\mathbb{R}^d}
\xi \,
p(\eta)
\exp
\big(\tfrac{\|\eta\|^2}{2\beta}
\big)
\exp
\big(-\tfrac{\|\xi\|^2}{2}
\big)
\,
\mathrm{d}\xi \\
& = 
\tfrac{\beta^{d/2}}{C}
\int_{\mathbb{R}^d}
\tfrac{\eta - x}{\sqrt{(1-t)\beta}} 
\exp\big(
\tfrac{\|\eta\|^2}{2\beta}
-\tfrac{\|\eta - x\|^2}{2{(1-t)\beta}}
\big)
\cdot
p(\eta)
\tfrac{\mathrm{d}\eta}{((1-t)\beta)^{d/2}} \\
& =
\tfrac{\beta^{d/2}}
{C((1-t)\beta)^{(d+1)/2}}
\mathbb{E}_{\eta \sim \mu } 
\left[ (\eta-x)  
\exp \big(
\tfrac{\|\eta\|^2}{2\beta}
-\tfrac{\|\eta-x\|^2}{2(1-t)\beta}
\big)
\right],
\end{aligned}
\end{equation}
and 
\begin{equation}
\begin{aligned}
\mathbb{E}_{\xi \sim \gamma^{d}} 
\left[  g_{\beta}
(x+\sqrt{1-t} \, \xi ) \right] 
=& 
\tfrac{\beta^{d/2}}{C((1-t)\beta)^{d/2}}
\mathbb{E}_{\eta \sim \mu } 
\left[  
\exp \big(
\tfrac{\|\eta\|^2}{2\beta}
-\tfrac{\|\eta-x\|^2}{2(1-t)\beta}
\big)
\right] .
\end{aligned}
\end{equation}
Equipped with the above equations, the drift $f_\beta$ of Schr\"odinger-F\"ollmer process with temperatures (\ref{eq-sfs-exact-dens}) can be rewritten as follows: 
\begin{equation}
\begin{aligned}
\label{eq-sfs-drift-iid}
    f_\beta(t,x)
    &= 
    \frac
    { \mathbb{E}_{\eta \sim \mu} 
    \left[ (\eta-x)  
    \exp \big(
    \tfrac{\|\eta\|^2}{2\beta}
    -\tfrac{\|\eta-x\|^2}{2(1-t)\beta}
    \big)
    \right]}
    { (1-t)\mathbb{E}_{\eta\sim \mu} 
    \left[             
    \exp \big(
    \tfrac{\|\eta\|^2}{2\beta}
    -\tfrac{\|\eta-x\|^2}{2(1-t)\beta}
    \big) 
    \right] }.
\end{aligned}
\end{equation}
This drift offers a powerful advantage in the data-driven setting 
because the expectations are taken with respect to the target distribution $\mu$.
More precisely, for $\mu$ being any given distribution in $\R^d$, and  $\{\eta^{(i)}\}^{n}_{i=1} \sim \mu$.
The drift $f_{\beta}(t,x) $
can be replaced by its empirical counterpart 
$\hat{f}_{\beta}^{M}$:
\begin{equation}
\begin{aligned}
\label{eq-sfs-mc-iid}
\hat{f}_{\beta}^{M}(t,x)
&= \frac{
\sum\limits_{j=1}^{M} \left[ (\eta^{(j)} - x)\cdot
\exp \big(
\tfrac{\|\eta^{(j)}\|^2}{2\beta}
-\tfrac{\|\eta^{(j)}-x\|^2}{2(1-t)\beta}
\big)
\right]}
{(1-t)\sum\limits_{j=1}^{M}
\left[
\exp \big(
\tfrac{\|\eta^{(j)}\|^2}{2\beta}
-\tfrac{\|\eta^{(j)}-x\|^2}{2(1-t)\beta}
\big)
\right]},
\quad \eta^{(j)} \sim \mu.
\end{aligned}
\end{equation}
Using this approximation, we propose the following data-driven sampler: 
\begin{equation}
\label{eq:sfs-data}
\begin{aligned}
\hat{Y}_{n+1}
&= 
\hat{Y}_{n} 
+ 
\tfrac{1}{3}
\hat{f}_{\beta}^{M}
(t_n, \hat{Y}_{n})h 
+ 
\tfrac{2}{3}
\hat{f}_{\beta}^{M}
\big(t_n+\tfrac{3}{4}h, \hat{H}_{n}\big)h+
\sqrt{\beta}
\Delta {W_n},
\end{aligned}
\end{equation}
where
\[
\hat{H}_{n} := \hat{Y}_{n} + \tfrac{3}{4}\hat{f}_{\beta}^{M}(t_n, \hat{Y}_{n})h + \tfrac{3\sqrt{\beta}\Delta {Z_n}}{2h}.
\]
We mention that this sampler does not require the training of the neural network to get sufficiently accurate score estimation and avoids the complexity
of building the network architecture, which is computationally more inexpensive.

In practice, 
the temperature parameter $\beta$ can be tuned to balance between exploration and exploitation. 
Larger values of $\beta$ introduce greater noise, 
facilitating exploration of the state space, 
while smaller values make the process more deterministic, 
focusing on mode-seeking behavior. 
The data-driven formulation provides a flexible framework for sampling from complex distributions using only empirical observations.

\subsection{Numerical experiments on data-driven sampling}

To validate the efficiency and sample quality of the proposed sampling framework 
in practical scenarios where explicit density forms of $\mu$ are unavailable, 
we conduct numerical experiments on data-driven problems.
In this subsection, we directly leverage the provided i.i.d. samples to generate new samples 
without intermediate density estimation or training of a score function. 
Our experiments evaluate the ability of the algorithm to capture complex data structures and produce high-quality samples directly from empirical distributions.

\paragraph{Low-dimensional distributions.}
First, we focus on two-dimensional "moons" and three-dimensional "S-curve" datasets.
The first example is sampled from a complicated distribution whose support is split into two disjoint regions of equal mass shaped like half-moon 
and the second example is a three-dimensional structure that forms a continuous, nonlinear manifold shaped like a twisted "S" \cite{pedregosa2011scikit}.
We generate $1000$ points as training samples, 
then apply the SRKSFS \eqref{eq:sfs-data} with a fixed step size $h=10^{-2}$ to generate new samples that resemble the original distribution. 
Figure \ref{fig-moons} and Figure \ref{fig-Scurve} visualize the results for the half-moon and S-curve datasets, respectively.
\begin{figure}[htbp]
    \centering
    \includegraphics[width=0.45\linewidth]{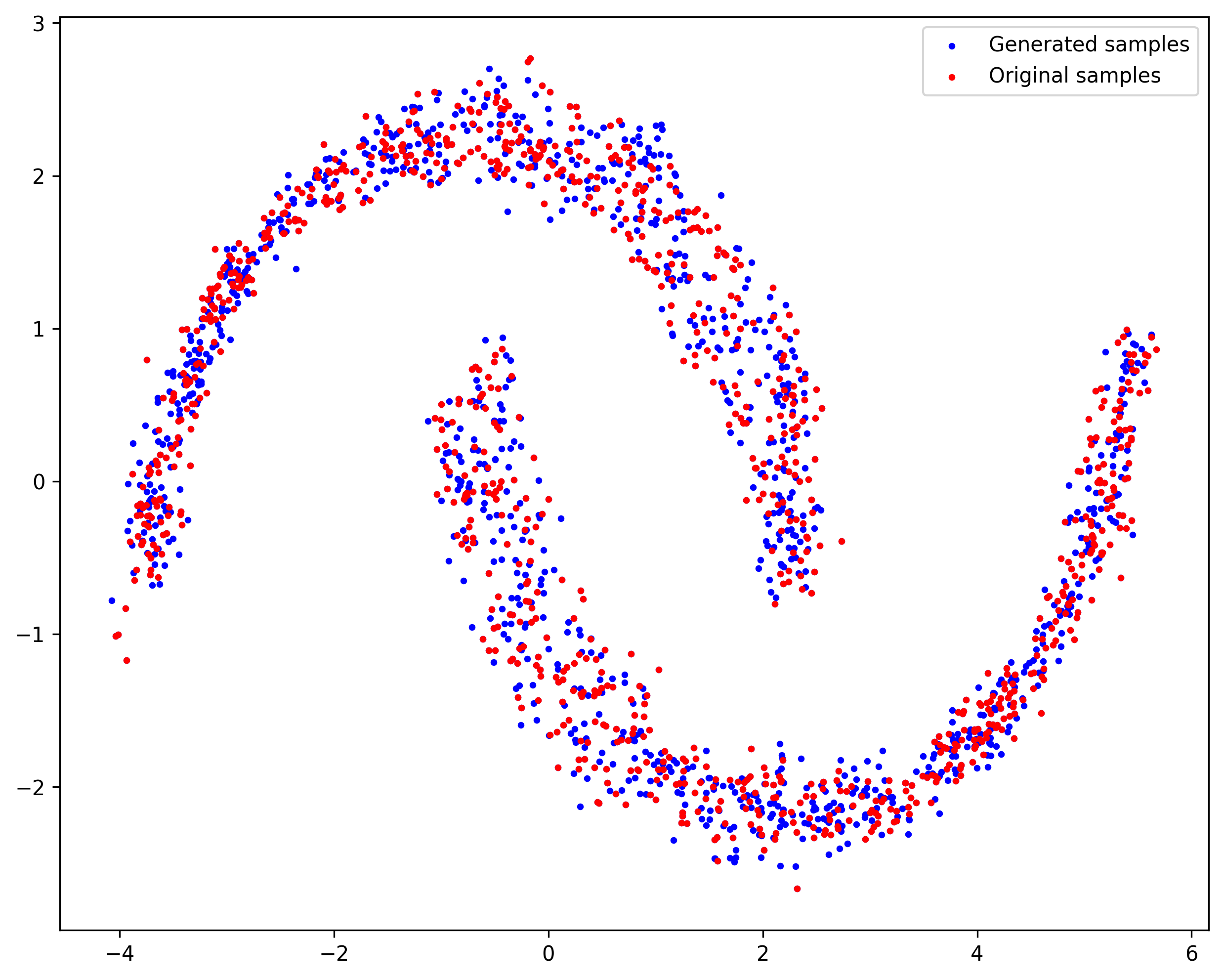}
    \caption{Sampling from Moons with SRKSFS.}
    \label{fig-moons}
\end{figure}

\begin{figure}[htbp]
    \centering
    \includegraphics[width=\linewidth]{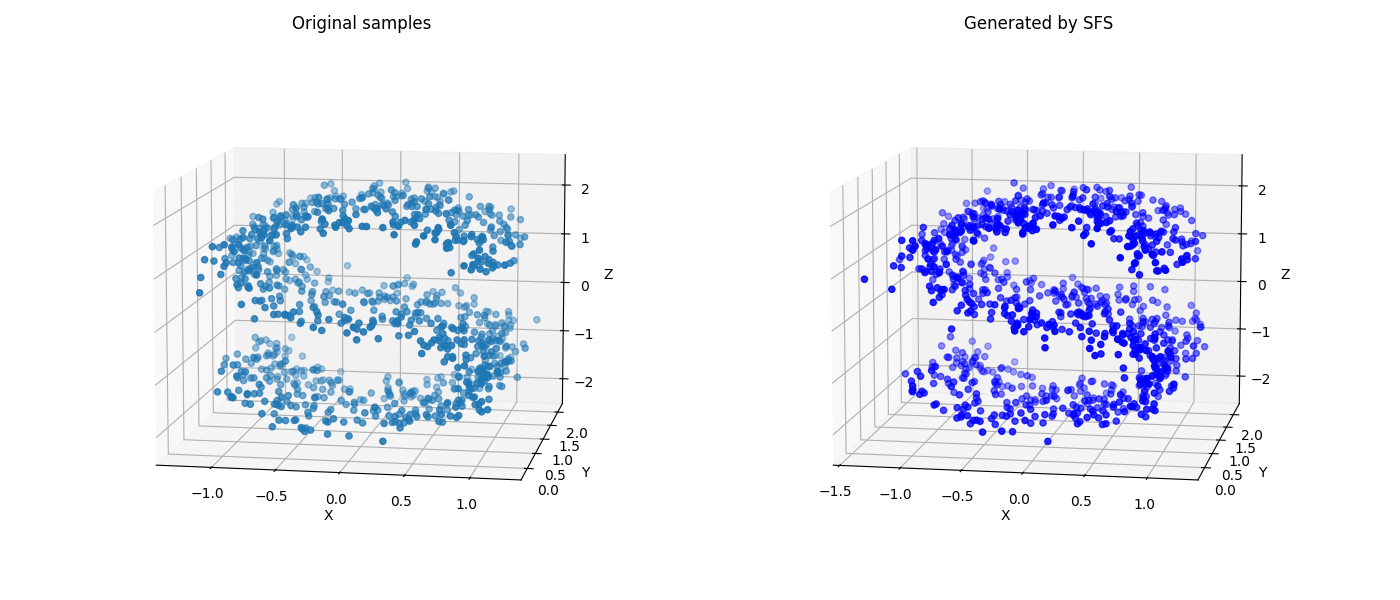}
    \caption{
    Sampling from S-curve with SRKSFS.}
    \label{fig-Scurve}
\end{figure}

\paragraph{Image generation from empirical distributions.}
To assess the scalability of our method in high-dimensional spaces, 
we consider image generation tasks using the MNIST \cite{deng2012mnist} and CIFAR-10 datasets \cite{krizhevsky2009learning}. 
In this setting, the target distribution is defined entirely by the training images, 
and our framework generates novel samples directly from the empirical distribution without density estimation or latent space normalization.

The MNIST data set consists of gray-valued digital images, 
each with $28 \times 28$ pixels showing one handwritten digit. 
The generated images for MNIST are presented on the right of Figure \ref{fig-mnist}, 
demonstrating that our method can synthesize high-fidelity images. 
\begin{figure}[htbp]
    \centering
    \includegraphics[width=\linewidth]{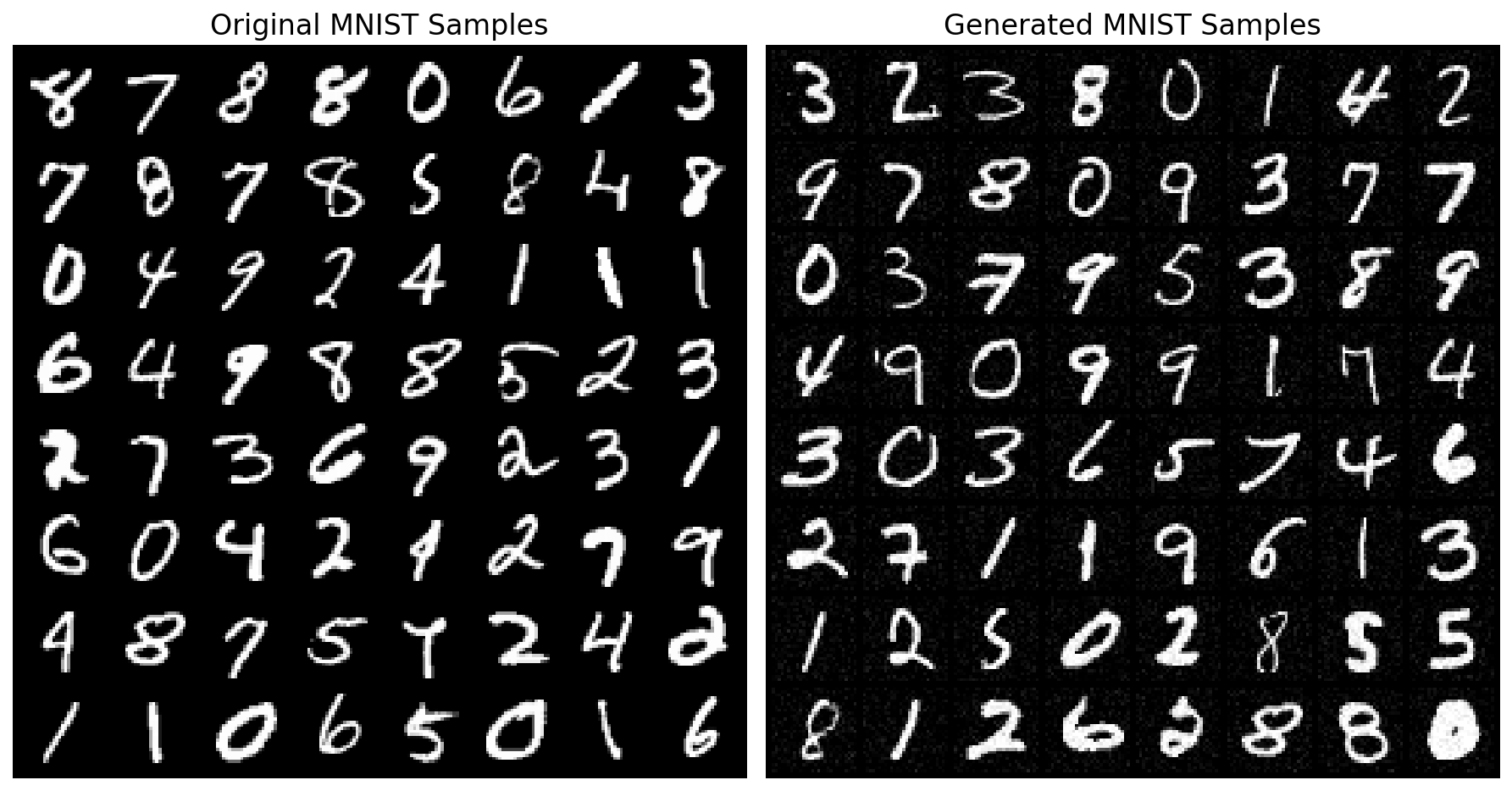}
    \caption{Samples and data from MNIST.}
    \label{fig-mnist}
\end{figure}
The CIFAR10 data set consists of $60000$ images in $10$ classes, with $6000$ images per class. 
The images are colored and of size $32 \times 32$ pixels. 
Figure \ref{fig-cifar10} shows generated samples (in right) alongside true samples (in left) from CIFAR-10. As the figure demonstrates, our algorithm successfully generates sharp, high-quality, 
and diverse samples in this high-dimensional image space, 
highlighting the effectiveness of direct sample-to-sample generation.
\begin{figure}[htbp]
    \centering
    \includegraphics[width=\linewidth]{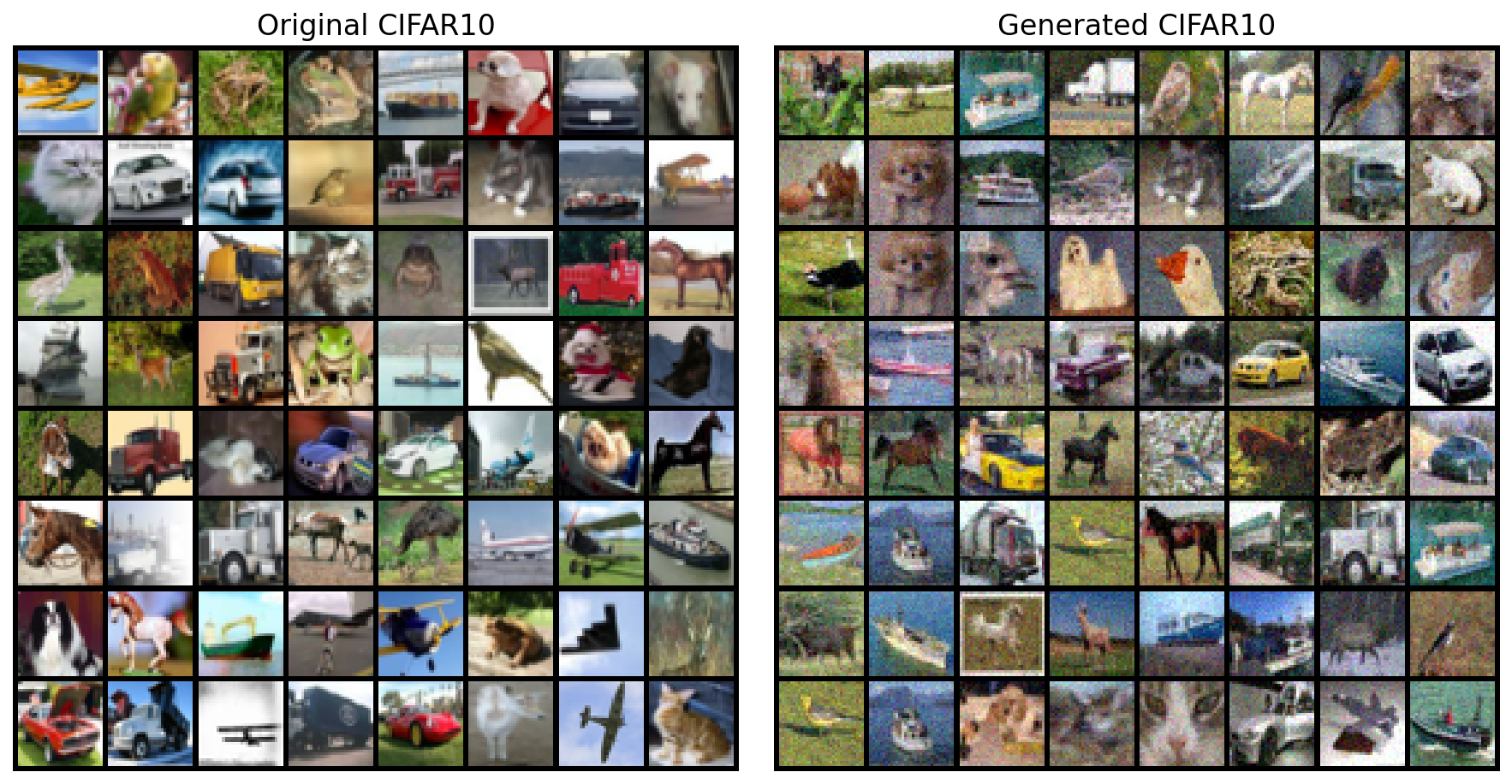}
    \caption{Samples and data from CIFAR-10.}
    \label{fig-cifar10}
\end{figure}


\section{Conclusion and future work}
\label{sec:Conclusion}
In this work, 
we introduce stochastic Runge-Kutta Schr{\"o}dinger-F{\"o}llmer samplers. 
Under mild smoothness assumptions on the drift, a convergence rate of order $\mathcal{O}((dh)^{3/2}| \ln h|)$ is established in $L^2$-Wasserstein distance, 
thereby providing a 
significant improvement over the Schr{\"o}dinger-F{\"o}llmer sampler based on Euler-Maruyama discretization  discussed in recent studies \cite{huang2025schrodinger,wang2025multimodal}.
%
%
As a future work, integrating the proposed higher-order discretization scheme with adaptive time-stepping strategies or multi-level Monte Carlo techniques could give further improvements in computational efficiency.


\bibliography{ref}
\appendix
\section{Proof of Proposition \ref{prop:g-assumption}}
\label{app:g-condition}
\begin{proof}
By assumptions, the function $g_\beta$ is of class $\mathcal{C}^4$, 
and moreover, 
$g_\beta, \nabla g_\beta, \nabla^2 g_\beta, \nabla^3 g_\beta$ are  all Lipschitz continuous.
Consequently, there exists a constant $L_g > 0$,
such that for any $x, y \in \mathbb{R}^d$,
\begin{equation}
   \| \nabla^k g_\beta(x)- \nabla^k g_\beta(y)\|
   \leq L_g \|x-y\|,
   \quad k=0,1,2,3.
\end{equation}
Following a similar argument as that in \cite[Appendix A]{wang2025multimodal}, we obtain, for any $x,v_1,v_2,v_3 \in \mathbb{R}^d$ 
and $t\in[0,1]$,
\begin{align*}
\big\|
D^3f_\beta(t,x)[v_1,v_2,v_3]
\big\|
& \leq  
\tfrac{\beta L_g}{\rho}  \big(
1+\tfrac{7L_g}{\rho}
+\tfrac{12L_g^2}{\rho^2}+\tfrac{6L_g^3}{\rho^3}\big).
\end{align*}
For the sake of notation, we denote $Z:=x+\sqrt{(1-t)\beta}\xi$ and $ Q(t,x):=Q_{1-t}^\beta g_{\beta}(x)$.
Noting that
\begin{align*}
    \partial_t 
    f_\beta (t,x) 
    = 
    -\tfrac{\beta^2}{2\sqrt{(1-t)\beta}}
    \left(
    \tfrac{\E[\nabla^2 g_\beta(Z)\cdot\xi]}{Q(t,x)}
    -
    \tfrac{\E[\nabla g_\beta(Z)] 
    \cdot
    \E
    \big[
    \langle\nabla g_\beta(Z),\xi\rangle
    \big]}{Q^2(t,x)}
    \right),
\end{align*}
one obtains the second time derivative as
\begin{align*}
    \partial_{tt} f_\beta (t,x)
    &= \tfrac{\beta^{3/2}}{2(1-t)^{3/2}}
    \left(
    \tfrac{\E[\nabla^2 g_\beta(Z)\cdot\xi]}{Q(t,x)}
    -
    \tfrac{\E[\nabla g_\beta(Z)] \cdot
    \E[\langle\nabla g_\beta(Z),\xi\rangle]}{Q^2(t,x)}
    \right)
    \\
    &\quad 
    -\tfrac{\beta^2}{2\sqrt{(1-t)\beta}}
    \bigg(
    \tfrac{\partial_t \E[\nabla^2 g_\beta(Z) \cdot \xi]}{Q(t,x)}
    -
    \tfrac{\E[\nabla^2 g_\beta(Z) \cdot \xi] \cdot \partial_t Q(t,x)}{Q(t,x)}
    \\
    &\qquad\quad
    -\tfrac{\partial_t \E[\nabla g_\beta(Z)]
    \cdot \E[\langle \nabla g_\beta(Z), \xi \rangle] 
    +\E[\nabla g_\beta(Z)] \cdot
    \partial_t \E[\langle \nabla g_\beta(Z), \xi \rangle]}
    {Q^2(t,x)}
    \\
    &\qquad\quad
    +\tfrac{2 \E[ \nabla g_\beta(Z)]
    \cdot
    \E[ \langle \nabla g_\beta(Z), \xi \rangle]
    \cdot
    Q(t,x) \cdot \partial_t Q(t,x)} {Q^4(t,x)}
    \bigg).  
\end{align*}
We have the following bounds
\begin{equation}
    \|\partial_t \E[\nabla^2 g_\beta(Z) \cdot \xi]\|
    =\|-\tfrac{\beta}{2\sqrt{(1-t)\beta}}\E[\nabla^3 g_\beta(Z) \cdot \xi \otimes \xi]]\|
    \leq \tfrac{\beta L_g d}{2\sqrt{(1-t)\beta}}.
\end{equation}
Similarly, we have
\begin{equation}
\begin{aligned}
    \|\E[\nabla^2 g_\beta(Z) \cdot \xi] \cdot \partial_t Q(t,x)\|
    &\leq \tfrac{\beta L_g^2 d}{2\sqrt{(1-t)\beta}},
    \\
    \|\partial_t \E[\nabla g_\beta(Z)]\|
    &\leq \tfrac{\beta L_g \sqrt{d}}{2\sqrt{(1-t)\beta}},
    \\
    \|\partial_t \E[\langle \nabla g_\beta(Z), \xi \rangle]\|
    &\leq \tfrac{\beta L_g d}{2\sqrt{(1-t)\beta}}.
\end{aligned}
\end{equation}
These estimates imply
\begin{equation}
    \|\partial_{tt} f_\beta (t,x)\|
    \leq 
    \tfrac{\beta^2 L_g}{2\rho}
    \big(2+\tfrac{4L_g}{\rho}+\tfrac{2L_g^2}{\rho^2}\big)
    \tfrac{d}
    {\sqrt{(1-t)^3}}.
\end{equation}
In the same manner, we can easily get the estimate \eqref{eq:mixed-diff}.
\end{proof}

\section{Proof of Proposition \ref{prop:mixing-time}}
\label{appendix_mixing-time}
\begin{proof}
Given an error tolerance $\epsilon>0$, Theorem \ref{thm-mc-error-bound} tells that, for $M$ being large enough and $h$ being small enough such that
\begin{align}
\label{eq:mixing-time-two-term}
    C (d h)^{3/2} 
    | \ln h|
    \leq 
    \tfrac{\epsilon}{2}, 
    \quad 
    C \sqrt{\tfrac{d}{M}}
    \leq 
    \tfrac{\epsilon}{2},
\end{align}
one can arrive at
\begin{align*}
        \mathcal{W}_2
        \big(
        \mathcal{L}aw
        (
        \widetilde{Y}_1
        ), 
        \mu
        \big)
        \leq 
        \epsilon.
\end{align*}
Rearranging the first inequality of \eqref{eq:mixing-time-two-term} gives
\begin{align*}
        \tfrac{(1/h)^3}{(\ln(1/h))^2} \geq \tfrac{d^3}{\epsilon^2}.
\end{align*}
Noting that the inequality $\tfrac{x^3}{(\ln x)^2} \geq K$ holds true on the condition $x \geq K^{1/3} (\ln K)^{2/3}$ for $x \geq 1$ , $ K > 0$, the above inequality is satisfied as
\begin{align*}
        N : =
        \tfrac{1}{h} \geq \tfrac{d}{\epsilon^{2/3}}\big(\ln\tfrac{d^{3/2}}{\epsilon}\big)^{2/3}.
\end{align*}
The second inequality of \eqref{eq:mixing-time-two-term} requires
\begin{align*}
    M 
    \geq 
    \tfrac{4C^2d}{\epsilon^2}.
\end{align*}
This completes the proof of the proposition.
\end{proof}
\end{document}